\documentclass[reqno]{amsart}
\usepackage{amsaddr}

\usepackage[english]{babel}
\usepackage[utf8]{inputenc}

\usepackage{amsthm,amsfonts,amssymb,amsmath}
\usepackage{stmaryrd}
\numberwithin{equation}{section}
\usepackage{dsfont}
\usepackage{multirow}
\usepackage{url}
\usepackage{algorithmic,algorithm}
\usepackage{graphics}
\usepackage{booktabs}
\usepackage{epsfig}
\usepackage{epstopdf}
\usepackage{psfrag}               
\usepackage{mathrsfs}
\usepackage{mathtools}            
\usepackage{subfigure}
\usepackage{longtable}	

\usepackage{paralist}
\usepackage[inline]{enumitem}

\usepackage[pdftex,
            pdftitle={Weak convergence for fractional SPDEs with spatial white noise},
            pdfauthor={David Bolin and K. Kirchner and M. Kovacs},
            bookmarksopen,
            colorlinks,
            linkcolor=black,
            urlcolor=black,
	    citecolor=black
]{hyperref}

\newcommand{\bbE}{\mathbb{E}}

\newcommand{\bbN}{\mathbb{N}}
\newcommand{\bbP}{\mathbb{P}}

\newcommand{\bbR}{\mathbb{R}}

\newcommand{\bbZ}{\mathbb{Z}}

\newcommand{\cA}{\mathcal{A}}

\newcommand{\cD}{\mathcal{D}}
\newcommand{\cE}{\mathcal{E}}
\newcommand{\cF}{\mathcal{F}}

\newcommand{\cL}{\mathcal{L}}

\newcommand{\cN}{\mathcal{N}}

\newcommand{\cW}{\mathcal{W}}

\newcommand{\scrD}{\mathscr{D}}

\newcommand{\bvec}{\mathbf{b}}
\newcommand{\gvec}{\mathbf{g}}
\newcommand{\jvec}{\mathbf{j}}
\newcommand{\xvec}{\mathbf{x}}
\newcommand{\zvec}{\mathbf{z}}
\newcommand{\xivec}{\boldsymbol{\xi}}
\newcommand{\phivec}{\boldsymbol{\phi}}

\newcommand{\Bmat}{\mathbf{B}}
\newcommand{\Gmat}{\mathbf{G}}
\newcommand{\Mmat}{\mathbf{M}}
\newcommand{\Smat}{\mathbf{S}}
\newcommand{\Qmat}{\mathbf{Q}}
\newcommand{\Rmat}{\mathbf{R}}
\newcommand{\Imat}{\mathbf{I}}

\newcommand{\T}{T} 

\newcommand{\norm}[2]{     \| #1       \|_{ #2 }}

\newcommand{\normiii}[2]{\vert\kern-0.25ex\vert\kern-0.25ex\vert #1 \vert\kern-0.25ex \vert\kern-0.25ex\vert_{ #2 }}
\newcommand{\Normiii}[2]{\left\vert\kern-0.25ex\left\vert\kern-0.25ex\left\vert #1 \right\vert\kern-0.25ex\right\vert\kern-0.25ex\right\vert_{ #2 }}
\newcommand{\scalar}[2]{     ( #1       )_{ #2 }}

\newcommand{\duality}[2]{     \langle #1       \rangle_{ #2 }}

\newcommand{\white}{\cW}

\newcommand{\tr}{\operatorname{tr}}
\newcommand{\rd}{\,\mathrm{d}}
\newcommand{\from}{\colon}
\newcommand{\Nok}{N_{\mathrm{ok}}}

\newcommand{\proj}{\Pi_h}
\newcommand{\ptilde}{\widetilde{\Pi}_h}
\newcommand{\Pbeta}{P_h^\beta}
\newcommand{\Qtilde}{\widetilde{Q}_{h,k}^\beta}
\newcommand{\QtildeT}{\widetilde{Q}_{h,k}^{\beta *}}
\newcommand{\Lhtilde}{\widetilde{L}_{h}^{-\beta}}
\newcommand{\LhtildeT}{\widetilde{L}_{h}^{-\beta *}}

\newcommand{\Hdot}[1]{\dot{H}^{#1}}

\graphicspath{{Images/}}

\newtheorem{lemma}{Lemma}[section]

\newtheorem{theorem}[lemma]{Theorem}

\theoremstyle{remark}
\newtheorem{remark}[lemma]{Remark}

\theoremstyle{definition}

\newtheorem{assumption}[lemma]{Assumption}
\newtheorem{example}[lemma]{Example}

\begin{document}

\title[Weak convergence for fractional elliptic SPDEs]
	{Weak convergence of Galerkin approximations for fractional elliptic stochastic PDEs with spatial white noise}


\author[D.~Bolin, K.~Kirchner, and M.~Kov\'{a}cs]{David Bolin, Kristin Kirchner, and Mih\'{a}ly Kov\'{a}cs}

\address[D.~Bolin, K.~Kirchner, and M.~Kov\'{a}cs]{
Department of Mathematical Sciences\\
Chalmers University of Technology and University of Gothenburg\\
412 96 G\"oteborg, Sweden}

\email{{\normalfont{(K.~Kirchner, corresponding author) }}kristin.kirchner@chalmers.se}


\thanks{
Acknowledgement.
This work was supported in part by the
Swedish Research Council (grant Nos.~2016-04187, 2017-04274),
and the Knut and Alice Wallenberg Foundation
(KAW 20012.0067).
The authors thank Stig Larsson for valuable comments on the manuscript and an
anonymous referee who helped to improve the presentation of the results.
}


\begin{abstract}
The numerical approximation
of the solution to
a stochastic partial differential
equation with additive spatial white noise
on a bounded domain
is considered.
The differential operator is assumed
to be a fractional power of
an integer order elliptic differential operator.
The solution is approximated by means
of a finite element discretization
in space and a quadrature approximation
of an integral representation
of the fractional inverse
from the Dunford--Taylor calculus.

For the resulting approximation,
a concise analysis of the weak error
is performed. Specifically, for the class of
twice continuously Fr\'echet differentiable
functionals
with second derivatives of polynomial growth,
an explicit rate of weak convergence
is derived, and it is shown that the component of
the convergence rate
stemming from the stochasticity
is doubled compared to the corresponding strong rate.
Numerical experiments
for different functionals
validate the theoretical results.
\end{abstract}



\date{\today}

\maketitle

\section{Introduction}\label{section:intro}

The representation of Gaussian random fields as solutions
to stochastic partial differential equations (SPDEs)
has become a popular approach in spatial statistics
in recent years.
It was observed already in~\cite{whittle54}
and \cite{whittle63} that
a Gaussian random field $u$ on $\bbR^d$
with a covariance function of Mat\'ern type \cite{matern60}
solves an SPDE of the form
$(\kappa^2 - \Delta)^\beta u = \white$.
Here, $\white$ is Gaussian white noise,
$\kappa>0$ is a parameter determining
the practical correlation range of the field,
and $\beta>d/4$ controls the smoothness parameter $\nu$
of the Gaussian Mat\'ern field via the equality $\nu = 2\beta - d/2$.

Later, this relation was the incentive
to consider the SPDE
\begin{align}\label{e:matern}
        (\kappa^2 - \Delta)^\beta u = \white \quad \text{in } \cD
\end{align}
for Gaussian random field approximations of Mat\'ern fields
on bounded domains $\cD\subsetneq\bbR^d$.
On the boundary $\partial \cD$, the operator $\kappa^2-\Delta$ is augmented
with, e.g., homogeneous Dirichlet or Neumann boundary conditions.
In \cite{lindgren11} it was shown that
by restricting the value of
$\beta$ to $2\beta\in\mathbb{N}$ and by solving the
stochastic problem~\eqref{e:matern}
by means of a finite element method,
the computational costs of many operations, which are needed
for statistical inference, such as sampling and likelihood evaluations
can be significantly reduced.
This decrease in computing time is one of the main reasons
for the popularity of the SPDE approach in spatial statistics.
In addition, it facilitates various extensions of the Mat\'ern model
which are difficult to formulate
using a covariance-based approach,
see, for instance, \cite{bolin14,bolin11,fuglstad2015,lindgren11,wallin15}.

However, the constraint $2\beta\in\mathbb{N}$ imposed by \cite{lindgren11}
restricts the value of the smoothness parameter $\nu$,
which is the most important parameter when the model is used for prediction \cite{stein99}.
In \cite{bkk17} we showed that this restriction can be avoided
by combining a finite element discretization in space
with a quadrature approximation
based on an integral representation
of the inverse fractional power operator 
from the Dunford--Taylor calculus.
We furthermore derived an explicit rate of convergence
for the strong mean-square error
of the proposed approximation
for a class of fractional elliptic stochastic equations
including \eqref{e:matern}.

In practice,
it is often not only necessary to sample
from the solution $u$ to \eqref{e:matern},
but also to estimate
the expected value $\bbE[\varphi(u)]$ of
a certain real-valued quantity of interest $\varphi(u)$.
The aim of this work is to provide
a concise analysis
of the weak error $|\bbE[\varphi(u)] - \bbE[\varphi(u_{h,k}^Q)]|$
for the approximation $u_{h,k}^Q$ proposed in \cite{bkk17}.
This analysis includes
the derivation of an explicit weak convergence rate
for twice continuously Fr\'{e}chet differentiable
real-valued functions $\varphi$, whose second derivatives
are of polynomial growth.
Functions of this form occur in many applications,
e.g., when integral means of the solution with respect
to a certain subdomain of $\cD$ are of interest,
or when a transformation of the model is used
as a component in a hierarchical model.
An example of the latter situation is
to consider logit or probit transformed Gaussian random fields
for binary regression models \cite[\S4.3.3]{rue2005gaussian}.

We prove that,
compared to the convergence rate of the strong error formulated in \cite{bkk17},
the component of the weak convergence rate stemming from the
stochasticity of the problem is doubled.
To this end,
two time-dependent stochastic processes are introduced,
which at time $t=1$ have the same probability distribution
as the exact solution $u$ and the approximation $u_{h,k}^Q$, respectively.
The weak error is then bounded
by introducing an associated Kolmogorov backward equation
on the interval $[0,1]$ and applying It\^{o} calculus.

The structure of this article is as follows:
In \S\ref{section:main} we formulate the
equation of interest in
a Hilbert space setting similarly to \cite{bkk17}
and state our main result
on weak convergence of the approximation
in Theorem~\ref{thm:weak-conv}.
A detailed proof of Theorem~\ref{thm:weak-conv}
is given in \S\ref{sec:derivation}.
For validating the theoretical result in practice,
we describe the outcomes of several
numerical experiments in \S\ref{section:numexp}.
Finally, \S\ref{section:conclusion} concludes the article with a discussion.

\vspace{2\baselineskip}

\section{Weak approximations}\label{section:main}

The subject of our investigations is the fractional order equation considered in~\cite{bkk17},
\begin{align}\label{e:Lbeta}
    L^\beta u = g + \white,
\end{align}
for $\beta\in(0,1)$, where $\white$ denotes Gaussian
white noise defined on a complete probability space
$(\Omega, \cA, \bbP)$ with values in a separable Hilbert space $H$.
Here and below, (in-)equalities involving random terms are meant
to hold $\bbP$-almost surely, if not specified otherwise.
Furthermore, we use the notation $X\overset{d}{=}Y$
to indicate that two random variables $X$ and $Y$ have the same
probability distribution.

Similarly to~\cite{bkk17}, we make the following assumptions:
$L\from\scrD(L) \subset H \to H$
is a densely defined, self-adjoint, positive definite operator 
and has a compact inverse $L^{-1}\from H \to H$.
In this case, $-L$ generates an analytic strongly continuous
semigroup $(S(t))_{t\geq 0}$ on $H$.
The $H$-orthonormal eigenvectors of $L$ are denoted by $\{e_j\}_{j\in\bbN}$
and the corresponding eigenvalues by $\{\lambda_j\}_{j\in\bbN}$.
These values are listed in nondecreasing order
and we assume that there exist constants
$\alpha, c_\lambda, C_\lambda > 0$ such that
\begin{align}\label{e:ass:lambdaj}
    c_\lambda \, j^\alpha \leq \lambda_j \leq C_\lambda \, j^\alpha
    \qquad \forall j\in\bbN.
\end{align}

The action of the fractional power operator $L^\beta$ in~\eqref{e:Lbeta}
is well-defined on
\begin{align*}
    \Hdot{2\beta} := \scrD(L^{\beta})
        = \Bigl\{ \psi\in H : \norm{\psi}{2\beta}^2 := \norm{L^\beta \psi}{H}^2 = \sum_{j\in\bbN} \lambda_j^{2\beta} \scalar{\psi, e_j}{H}^2 < \infty \Bigr\},
\end{align*}
which is itself a Hilbert space with inner product
$\scalar{\phi, \psi}{2\beta} := \scalar{L^{\beta} \phi, L^{\beta} \psi}{H}$.
Furthermore, there exists a unique continuous extension of $L^\beta$ to
an isometric isomorphism $L^\beta \from \Hdot{r} \to \Hdot{r-2\beta}$
for all $r\in\bbR$, see~\cite[Lem.~2.1]{bkk17}.
Here, for $s > 0$, the negative-indexed space $\Hdot{-s}$ is defined
as the dual space of $\Hdot{s}$. 
After identifying the dual space~$H^*$ of~$\Hdot{0} := H$
via the Riesz map, we obtain the Gelfand triple
$\Hdot{s} \hookrightarrow H \cong H^* \hookrightarrow \Hdot{-s}$
with continuous and dense embeddings.
The norm on the dual space $\Hdot{-s}$
can be expressed by
\begin{align*}
    \norm{g}{-s}
    =
    \sup\limits_{\phi\in\Hdot{s}\setminus\{0\}}
    \frac{\duality{g, \phi}{}}{\norm{\phi}{s}}
    =
    \biggl( \sum_{j\in\bbN} \lambda_j^{-s} \duality{g, e_j}{}^2 \biggr)^{\frac{1}{2}},
\end{align*}
where $\duality{\,\cdot\,,\,\cdot\,}{}$ denotes
the duality pairing on $\Hdot{-s}$ and
$\Hdot{s}$,~\cite[Proof~of~Lem.~5.1]{thomee2007}.
With this representation of the dual norm
and the growth~\eqref{e:ass:lambdaj} of the eigenvalues~$\lambda_j$ at hand,
it is an immediate consequence
of a Karhunen--Lo\`{e}ve expansion of the white noise $\white$
with respect to the $H$-orthonormal eigenvectors $\{e_j\}_{j\in\bbN}$ 
that $\white$ has mean-square regularity in $\Hdot{-s}$
for every $s > \alpha^{-1}$, see~\cite[Prop.~2.3]{bkk17}.
Consequently, \eqref{e:Lbeta} has a solution
$u \in L_2(\Omega; \Hdot{2\beta - s})$ for $s > \alpha^{-1}$ if $g\in\Hdot{-s}$.

\subsection{The Galerkin approximation}\label{subsec:fraceq}

In the following, let $(V_h)_{h\in(0,1)}$
be a family of
subspaces of $\Hdot{1}=\cD(L^{1/2})$
with finite dimensions
$N_h := \dim(V_h)$
and let $\proj \from H \to V_h$
be the $H$-orthogonal projection
onto $V_h$.
For $g\in H$, we define the finite element approximation
of $v = L^{-1}g$ by
$v_h = L_h^{-1} \proj g$, where $L_h$
denotes the Galerkin discretization
of the operator $L$ with respect to $V_h$,  i.e.,
\begin{align*}
L_h \from V_h \to V_h,
\qquad
\scalar{L_h\psi_h,\phi_h}{H} = \duality{L\psi_h,\phi_h}{}
\quad
\forall \psi_h,\phi_h \in V_h.
\end{align*}

We then consider the following numerical approximation
of the solution $u$ to~\eqref{e:Lbeta}
\begin{align}\label{e:uQh}
    u_{h,k}^{Q} := Q_{h,k}^\beta(\proj g + \white_h^\Phi)
\end{align}
proposed in~\cite[Eq.~(2.18)]{bkk17}.
It is  based on the following two components:
\begin{itemize}
\item[(a)] The operator $Q_{h,k}^\beta$ is the quadrature approximation for $L_h^{-\beta}$ of~\cite{bonito2015}:
    \begin{align}\label{e:def:Qhk}
        Q^\beta_{h,k} := \frac{2 k \sin(\pi\beta)}{\pi} \sum_{\ell=-K^{-}}^{K^{+}} e^{2\beta y_\ell} \left(\mathrm{Id}_{V_h} + e^{2 y_\ell} L_h\right)^{-1}.
    \end{align}
    The quadrature nodes
    $\{y_\ell = \ell k : \ell\in\bbZ, -K^{-} \leq \ell \leq K^+\}$
    are equidistant with distance $k>0$ and we set
    $K^- := \bigl\lceil \tfrac{\pi^2}{4\beta k^2} \bigr\rceil$
    and $K^+ := \bigl\lceil \frac{\pi^2}{4(1-\beta)k^2} \bigr\rceil$.
\item[(b)] The white noise $\white$ in $H$ is approximated by
    the square-integrable \linebreak $V_h$-valued random variable $\white_h^\Phi$ 
    given by $\white_{h}^{\Phi} := \sum_{j=1}^{N_h} \xi_j \, \phi_{j,h}$,
    where \linebreak $\Phi:=\{\phi_{j,h}\}_{j=1}^{N_h}$ is any basis
    of the finite element space $V_h$.
    The vector $\xivec = (\xi_1,\ldots, \xi_{N_h})^\T$ is multivariate Gaussian distributed
    with mean zero and covariance matrix $\Mmat^{-1}$,
    where $\Mmat$ denotes the mass matrix with respect to the basis $\Phi$,
    i.e., $M_{ij} = \scalar{\phi_{i,h}, \phi_{j,h}}{H}$.
\end{itemize}

The main outcome of~\cite{bkk17} is strong convergence 
of the approximation $u_{h,k}^{Q}$ in~\eqref{e:uQh} to the solution
$u$ of~\eqref{e:Lbeta} at an explicit rate.
Subsequently, this work focusses on weak approximations based on $u_{h,k}^{Q}$, i.e.,
we investigate the error
\begin{align}\label{e:def:err-weak}
    \bigl| \bbE[\varphi(u)] - \bbE[\varphi(u_{h,k}^Q)] \bigr|
\end{align}
for continuous functions $\varphi\from H \to \bbR$.

\begin{remark}
In practice, the expected value $\bbE[ \varphi(u_{h,k}^Q) ]$
is approximated, e.g., by a Monte Carlo method.
For this, usually a large number of realizations of $\varphi(u_{h,k}^Q)$ and, thus,
of the approximation $u_{h,k}^Q$ in~\eqref{e:uQh} is needed.
Each of them requires a sample of the load vector $\bvec$ with
entries $b_j := \scalar{\proj g + \white_h^\Phi, \phi_{j,h}}{H}$.
As pointed out in~\cite[Rem.~2.9]{bkk17}, this is computationally feasible
if the mass matrix $\Mmat$ with respect to the finite element basis $\Phi$
is sparse, since the distribution of $\xivec \sim \cN(\mathbf{0},\Mmat^{-1})$ implies that
\begin{align*}
    \bvec \sim \cN(\gvec, \Mmat),
    \qquad
    \bvec \overset{d}{=} \gvec + \Gmat \zvec,
\end{align*}
where $\zvec \sim \cN(\mathbf{0},\Imat)$,
$\Gmat$ is the Cholesky factor of $\Mmat = \Gmat \Gmat^\T$,
and the vector $\gvec$ has entries
$g_j := \scalar{g, \phi_{j,h}}{H}$.
\end{remark}

\subsection{Weak convergence}\label{subsec:weak-conv}

For bounding the error in~\eqref{e:def:err-weak},
we start by introducing some more notation and assumptions.
Let $\cE := \{e_{j,h}\}_{j=1}^{N_h} \subset V_h$
be the $H$-orthonormal eigenvectors
of the discrete operator $L_h$ with
corresponding eigenvalues $\{\lambda_{j,h}\}_{j=1}^{N_h}$
listed in nondecreasing order.
In addition,
the strongly continuous semigroup on $V_h$
generated by $-L_h$ is denoted by $(S_h(t))_{t\geq 0}$.

We define the space $C^2(H;\bbR)$
of twice continuously Fr\'{e}chet differentiable
functions $\varphi\from H \to \bbR$, i.e.,
$\varphi\in C^2(H;\bbR)$ if and only if
\begin{align*}
    \varphi \in C(H;\bbR),
    \quad
    D\varphi \in C(H;H),
    \quad
    \text{and}
    \quad
    D^2 \varphi \in C(H;\cL(H)) .
\end{align*}
Here and below,
using the Riesz representation theorem,
we identify the first
two Fr\'echet derivatives $D\varphi$ and $D^2 \varphi$
of $\varphi$ with functions taking values in $H$
and in $\cL(H)$, respectively.
Furthermore, we say that the second derivative
has polynomial growth of degree $p\in\bbN$, if there
exists a constant $K>0$ such that
\begin{align}\label{e:ass:phi-poly}
    \norm{D^2 \varphi (\psi)}{\cL(H)} \leq K \left( 1 + \norm{\psi}{H}^p \right)
    \quad
    \forall\psi\in H.
\end{align}

All the properties
of the finite element discretization,
of the operator $L$,
and of the function $\varphi$,
which are of importance for our analysis
of the weak error~\eqref{e:def:err-weak},
are summarized in the assumption below.
\begin{assumption}\label{ass:all}
The finite element spaces $(V_h)_{h\in(0,1)} \subset \Hdot{1}$,
the operator $L$ in~\eqref{e:Lbeta}, and
the function $\varphi\from H \to \bbR$ in~\eqref{e:def:err-weak}
satisfy the following:
\begin{enumerate}[label=(\roman*)]
\item\label{ass:Vh-1} there exists $d\in\bbN$ such that $N_h = \dim(V_h) \propto h^{-d}$ for all $h > 0$;
\item\label{ass:Vh-2} there exist constants $C_1, C_2 > 0$, $h_0\in(0,1)$,
    as well as exponents $r,s > 0$ and $q > 1$
    such that
    \begin{align*}
        \lambda_j \leq \lambda_{j,h} &\leq \lambda_j + C_1 h^r \lambda_j^q, \\
        \norm{e_j - e_{j,h}}{H}^2    &\leq C_2 h^{2s} \lambda_j^q,
    \end{align*}
    for all $h\in(0,h_0)$ and $j\in\{1,\ldots,N_h\}$;
\item\label{ass:L} the eigenvalues of $L$ satisfy \eqref{e:ass:lambdaj}
    for an exponent $\alpha$ with
    \begin{align*}
        \tfrac{1}{2\beta}
        <
        \alpha
        \leq
        \min\left\{ \tfrac{r}{(q-1)d}, \tfrac{2s}{q d} \right\},
    \end{align*}
    where the values of
    $d\in\bbN$, $r,s>0$, and $q>1$
    are the same as in~\ref{ass:Vh-1}--\ref{ass:Vh-2};
\item\label{ass:Sh}
    $s>2\beta$ and for $0\leq\theta\leq\sigma\leq s$
    there exists a constant $C_3 > 0$ such that
    \begin{align*}
        \norm{(S(t)-S_h(t)\proj)g}{H} \leq C_3 h^\sigma t^{\frac{\theta-\sigma}{2}} \norm{g}{\theta}
        \quad
        \forall t>0,
    \end{align*}
    for every $g\in\Hdot{\theta}$ and $h\in(0,h_0)$.
    Here, $h_0$ and $s$ are as in~\ref{ass:Vh-2};
\item\label{ass:phi} $\varphi \in C^2(H;\bbR)$ and $D^2 \varphi$ has
    polynomial growth~\eqref{e:ass:phi-poly} of degree $p\geq 2$.
\end{enumerate}
\end{assumption}

The following example shows that
Assumptions \ref{ass:all}\ref{ass:Vh-1}--\ref{ass:Sh}
are satisfied, e.g.,
for the motivating problem~\eqref{e:matern}
related to approximations of Mat\'ern fields, if $\beta > d/4$,
when using continuous piecewise linear
finite element bases.

\begin{example}\label{ex:matern}
For $\kappa \geq 0$ and a bounded, convex, polygonal domain $\cD\subset\bbR^d$,
consider the stochastic model problem~\eqref{e:matern}, 
i.e., the fractional order equation \eqref{e:Lbeta} for $g=0$ and 
$L = \kappa^2 - \Delta$
on $H=L_2(\cD)$.
Furthermore, we assume that
the differential operator~$L$ is augmented
with homogeneous Dirichlet boundary conditions on $\partial\cD$.
In this case, the eigenvalues $\{\lambda_j\}_{j\in\bbN}$ of $L$
satisfy~\eqref{e:ass:lambdaj} for $\alpha = 2/d$
(see \cite[Ch.~VI.4]{courant1962} for $\cD=(0,1)^d$,
the result for more general domains as above follows from the min-max principle).
Consequently, the first inequality of Assumption~\ref{ass:all}\ref{ass:L} holds
if $\beta > d/4$.

In addition,
if $(V_h)_{h\in(0,1)} \subset \Hdot{1} = H_0^1(\cD)$
are finite element spaces
with continuous piecewise linear basis functions
defined with respect to a quasi-uniform family of triangulations, 
Assumption~\ref{ass:all}\ref{ass:Vh-1} holds
and
Assumptions~\ref{ass:all}\ref{ass:Vh-2}, \ref{ass:Sh} are satisfied
for $r=s=q=2$,
see \cite[Thm.~6.1, Thm.~6.2]{strang2008}
and \cite[Thm.~3.5]{thomee2007}.
Thus, 
\begin{align*}
s = 2 > 2\beta,
\qquad
\alpha = \tfrac{2}{d}
        =
        \min\left\{ \tfrac{r}{(q-1)d}, \tfrac{2s}{q d} \right\},
\end{align*}
and
Assumptions \ref{ass:all}\ref{ass:Vh-1}--\ref{ass:Sh} hold
for all $\beta\in(d/4,1)$.
\end{example}

We remark that Assumptions~\ref{ass:all}\ref{ass:Vh-1}--\ref{ass:L} coincide
with those of~\cite{bkk17}. 
The strong $L_2(\Omega;H)$-convergence rate
\begin{align}\label{e:rate-strong}
    \min\{d(\alpha\beta - 1/2),r,s\}
\end{align}
was derived in~\cite[Thm.~2.10]{bkk17}
for the approximation $u_{h,k}^{Q}$ in~\eqref{e:uQh}
under a suitable calibration of the distance of the quadrature nodes $k$
with the finite element mesh size~$h$.
Furthermore, a bound for the weak-type error
\begin{align*}
    \bigl| \norm{u}{L_2(\Omega;H)}^2 - \norm{u_{h,k}^Q}{L_2(\Omega;H)}^2 \bigr|
\end{align*}
was provided, showing convergence to zero
with the rate $\min\{d(2\alpha\beta - 1),r,s\}$,
see \cite[Cor.~3.4]{bkk17}.
In particular, the term $d(2\alpha\beta - 1)$
stemming from the stochasticity is doubled
compared to the strong rate in~\eqref{e:rate-strong}.

In the following, we generalize this
result to weak errors of the form~\eqref{e:def:err-weak}
for functions $\varphi\from H \to \bbR$,
which are twice continuously Fr\'{e}chet differentiable
and have a second derivative of polynomial growth.
The bound of the weak error
in Theorem~\ref{thm:weak-conv}
is our main result.

\begin{theorem}\label{thm:weak-conv}
Let Assumption~\ref{ass:all} be satisfied.
Let $\theta > \min\{d(2\alpha\beta-1),s\} - 2\beta$,
if $d(2\alpha\beta-1) \geq 2\beta$,
and set $\theta = 0$ otherwise. %
Then, for $g\in\Hdot{\theta}$ and for
sufficiently small $h\in(0,h_0)$ and $k\in(0,k_0)$,
the weak error in~\eqref{e:def:err-weak}
admits the bound
\begin{align}
    \bigl| \bbE[\varphi(u)] - \bbE[\varphi(u_{h,k}^{Q})] \bigr|
         &\leq
            C \left( h^{\min\{d(2\alpha\beta-1),r,s\}} + e^{-\frac{\pi^2}{k}} h^{-d}
                + e^{-\frac{\pi^2}{2k}} + e^{-\frac{\pi^2}{2k}} f_{\alpha,\beta}(h) \right) \notag \\
            &\quad \times (1 + e^{-\frac{p\pi^2}{2k}} h^{-\frac{pd}{2}} + \norm{g}{H}^{p+1} ) (1 + \norm{g}{\theta} ).
            \label{e:err-weak}
\end{align}
Here, we set $f_{\alpha,\beta}(h) := h^{d(\alpha\beta - 1)}$, if $\alpha\beta\neq 1$, and
$f_{\alpha,\beta}(h) := |\ln(h)|$, if $\alpha\beta = 1$.
The constant $C>0$ is independent of $h$ and $k$
and the values of $\alpha,r,s > 0$, $d\in\bbN$, and $p\in\{2,3,\ldots\}$ are those of
Assumption~\ref{ass:all}.
\end{theorem}

\begin{remark}\label{remark:calibration}
In the derivation of the strong convergence
rate~\eqref{e:rate-strong}, 
we balanced the error terms caused by the quadrature
and by the finite element method by choosing
the quadrature step size $k$ sufficiently small
with respect to the finite element mesh width~$h$,
namely $e^{-\pi^2/(2k)} \propto h^{d\alpha\beta}$, see~\cite[Table~1]{bkk17}.

For calibrating the terms in the weak error estimate~\eqref{e:err-weak},
we distinguish the cases $\alpha\beta < 1$, $\alpha\beta = 1$, and $\alpha\beta > 1$. 
If $\alpha\beta < 1$, then $d\alpha\beta > d(2\alpha\beta-1)$
and we let $k>0$ be such that
$e^{-\pi^2/(2k)} \propto h^{d\alpha\beta}$.
With this choice, the error estimate~\eqref{e:err-weak} simplifies to
\begin{align*}
        \bigl| \bbE[\varphi(u)] - \bbE[\varphi(u_{h,k}^{Q})] \bigr|
         \leq C
             h^{\min\{d(2\alpha\beta-1),r,s\}}  (1 + \norm{g}{H}^{p+1} ) ( 1 + \norm{g}{\theta}).
\end{align*}
For $\alpha\beta > 1$ ($\alpha\beta=1$)  
we achieve the same bound if $k$ and $h$ are calibrated such that
$e^{-\pi^2/(2k)} \propto h^{d(2\alpha\beta-1)}$ ($e^{-\pi^2/(2k)} \max\{1,|\ln(h)|\} \propto h^d$).
Note that the calibration for $\alpha\beta < 1$
coincides with the one for the strong error
and that the term $d(2\alpha\beta-1)$ in the derived
weak convergence rate
$\min\{d(2\alpha\beta-1),r,s\}$
is doubled compared to the first term of the
strong convergence rate~\eqref{e:rate-strong}.
\end{remark}

\begin{remark}
We emphasize that (under the same assumptions)
both the strong and weak convergence rates
remain the same when approximating the solution $u$ to
\[
    L^\beta u = \sigma( g + \white)
\]
by $u_{h,k}^Q := \sigma \, Q_{h,k}^\beta (\Pi_h g + \white_h^\Phi)$,
where $\sigma > 0$ is a constant parameter which scales
the variance of $u$.
This can be seen from the equality
$\sigma^{-1} L^\beta = L_\sigma^{\beta}$ for $L_\sigma := \sigma^{-1/\beta} L$,
combined with the fact that the eigenvalues of the operator $L_\sigma$
satisfy the growth assumption~\eqref{e:ass:lambdaj}
with the same exponent $\alpha > 0$ as the eigenvalues of $L$.

However, the constants
$c_\lambda, C_\lambda > 0$
in~\eqref{e:ass:lambdaj}
and the constants in the error estimates change.
For instance, if $\varphi(u) := \norm{u}{H}^{p_{*}}$ for $p_{*}\in\bbN$,
then the constant $C>0$ in~\eqref{e:err-weak} will depend linearly on $\sigma^{p_{*}}$.

Note that one has to consider
a problem of the form
\[
    (\kappa^2 - \Delta)^\beta u = \sigma \white
    \quad
    \text{for}
    \quad
    \sigma := \sigma_{*} (4\pi)^{\frac{d}{4}} \kappa^{2\beta-\frac{d}{2}} \sqrt{\tfrac{\Gamma(2\beta)}{\Gamma(2\beta-d/2)}}
\]
when approximating a Mat\'ern field
with variance $\sigma_{*}^2$.
Here and in what follows, $\Gamma(\,\cdot\,)$
denotes the Gamma function.
\end{remark}

\begin{remark}\label{remark:rational}
We also comment on how the error bound in~\eqref{e:err-weak}
changes if instead of the family $(Q_{h,k}^\beta)_{k>0}$
a different sequence of approximations
$\{R_{h,n}^\beta\}_{n\in\bbN}$ of $L_h^{-\beta}$
is used.
If there exists a function $E\from \bbN \to \bbR_{\geq 0}$
such that $\lim_{n\to\infty} E(n) = 0$ as well as
a constant $C>0$, independent of $h$ and $n$,
such that
\begin{align*}
    \norm{(L_h^{-\beta} - R_{h,n}^\beta)\phi_h}{H}
        \leq C E(n) \norm{\phi_h}{H}
        \quad
        \forall \phi_h \in V_h,
\end{align*}
it is an immediate consequence of the arguments in our proof that
a bound of the weak error for the approximation
$u^R_{h,n} := R_{h,n}^\beta(\proj g + \white_h^\Phi)$
is given by
\begin{align*}
    \bigl| \bbE[\varphi(u)] - \bbE[\varphi(u_{h,k}^{Q})] \bigr|
         &\leq
            C \left( h^{\min\{d(2\alpha\beta-1),r,s\}} + E(n)^2 h^{-d} + E(n) + E(n) f_{\alpha,\beta}(h) \right)  \\
         &\quad \times (1 + E(n)^p h^{-\frac{pd}{2}} + \norm{g}{H}^{p+1} ) (1 + \norm{g}{\theta} ).
\end{align*}
An example of such a family $\{R_{h,n}^\beta\}_{n\in\bbN}$
are the approximations of $L_h^{-\beta}$
proposed in \cite{bolin2017fractional},
which are based on rational approximations
of the function $x^{-\beta}$ of different degrees $n\in\bbN$.
%
\end{remark}

\section{The derivation of Theorem~\ref{thm:weak-conv}}\label{sec:derivation}

The main idea in our derivation of the weak error estimate~\eqref{e:err-weak}
is to introduce two time-dependent stochastic processes
with the property that their
(random) values at time $t=1$
have the same distribution as
the solution
$u$ to~\eqref{e:Lbeta}
and its approximation~$u_{h,k}^Q$ in~\eqref{e:uQh}, respectively.
We then use an associated Kolmogorov backward equation and
It\^{o} calculus to estimate the difference
between these values.

\subsection{The extension to time-dependent processes}\label{subsec:time}

Recall the eigenvalue-eigen- vector pairs
$\{ (\lambda_j, e_j) \}_{j\in\bbN}$ of $L$
as well as the parameter $\alpha>0$
determining the growth of the eigenvalues
via~\eqref{e:ass:lambdaj}.
In what follows, we assume that $g\in H$
and $2\alpha\beta > 1$ so that the solution $u$
to~\eqref{e:Lbeta} satisfies $u\in L_2(\Omega; H)$.
With the aim of introducing the time-dependent processes mentioned above,
we start by defining
\begin{align*}
    W^{\beta}(t) := \sum_{j\in\bbN} \lambda_j^{-\beta} B_j(t) \, e_j,
    \quad
    t\geq 0,
\end{align*}
where $\{B_j\}_{j \in \bbN}$ is a sequence of independent real-valued Brownian motions
adapted to a filtration $\cF := (\cF_t, \ t\geq 0)$.
Owing to this construction, $(W^{\beta}(t), \ t\geq 0)$ is
an $\cF$-adapted $H$-valued Wiener process with covariance operator
$L^{-2\beta}$, which is of trace-class if $2\alpha\beta > 1$.
Since the random variables $\{B_j(1)\}_{j \in \bbN}$
are independent and identically $\cN(0,1)$-distributed,
the spatial white noise $\white$ satisfies
\begin{align*}
    \white \overset{d}{=} \sum_{j\in\bbN} B_j(1) \, e_j
    \quad
    \text{in }H.
\end{align*}

The stochastic process $Y := (Y(t), \ t \in [0,1])$ defined as
the (strong) solution to the stochastic partial differential equation
\begin{align}\label{e:spde-Y}
    \mathrm{d} Y(t) = \mathrm{d} W^{\beta}(t),
    \quad
    t \in [0,1],
    \qquad
    Y(0) = L^{-\beta} g,
\end{align}
therefore takes the following random value in $H$ at time $t=1$,
\begin{align}\label{e:u-Y}
    Y(1) = Y(0) + \int_0^1 \mathrm{d} W^{\beta}(t) = L^{-\beta} g + W^{\beta} (1) \overset{d}{=} L^{-\beta} (g + \white) = u.
\end{align}
%

Its Gaussian distribution
implies the existence of all moments,
as shown in the following lemma.

\begin{lemma}\label{l:Ymoments}
Let $p\in \bbN$, $t\in[0,1]$, and $Y$
be the strong solution
of~\eqref{e:spde-Y}.
Then the $p$-th moment of $Y(t)$
exists and, for $p \geq 2$,
it admits the following bound:
\begin{align}\label{e:Lp-Y}
    \bbE\left[ \norm{Y(t)}{H}^p \right]
        &\leq 2^{p-1} \left( \norm{g}{-2\beta}^p + t^{\frac{p}{2}} \mu_p \tr (L^{-2\beta})^{\frac{p}{2}}  \right).
\end{align}
Here, $\mu_p := \bbE[ |Z|^p ]
                    = \sqrt{\frac{2^p}{\pi}} \, \Gamma\left(\frac{p+1}{2}\right)$
is the $p$-th absolute moment
of $Z\sim\cN(0,1)$.
\end{lemma}

\begin{proof}
For $p=2$, the bound in~\eqref{e:Lp-Y}
follows from the It\^{o} isometry \cite[Thm.~8.7(i)]{peszat2007}:
\begin{align*}
    \bbE\left[ \norm{Y(t)}{H}^2 \right]
        &= \norm{L^{-\beta}g}{H}^2 + \int_0^t \tr(L^{-2\beta}) \rd s = \norm{g}{-2\beta}^2 + t \mu_2 \tr(L^{-2\beta}).
\end{align*}

If $p\geq 3$,  we estimate
$\bbE[ \norm{Y(t)}{H}^p ]
\leq 2^{p-1} (\norm{L^{-\beta}g}{H}^p + \bbE[ \norm{W^\beta(t)}{H}^p ])$.
By Jensen's inequality we have
\begin{align*}
     \bbE[ \norm{W^\beta(t)}{H}^p ] = \bbE \Bigl| \sum_{j\in\bbN} \lambda_j^{-2\beta} |B_j(t)|^2 \Bigr|^{\frac{p}{2}}
        \leq \bbE\Biggl[ \Bigl| \sum_{j\in\bbN} \lambda_j^{-2\beta} \Bigr|^{\frac{p}{2}-1}
            \sum_{j\in\bbN} \lambda_j^{-2\beta} |B_j(t)|^p \Biggr].
\end{align*}
Thus, the distribution of $\{B_j(t)\}_{j\in\bbN}$ implies that
$\bbE[ \norm{W^\beta(t)}{H}^p ] \leq t^{p/2} \mu_p \tr( L^{-2\beta} )^{p/2}$
and assertion~\eqref{e:Lp-Y} follows.
\end{proof}

In order to define a another stochastic process
$\widetilde{Y} := (\widetilde{Y}(t), \ t \in [0,1])$
with the property $\widetilde{Y}(1) \overset{d}{=} u_{h,k}^Q$ in $H$,
we recall the orthonormal eigenbasis $\cE = \{e_{j,h}\}_{j=1}^{N_h} \subset V_h$ of $L_h$ and define $\Pbeta\from H \to V_h$ for $\beta\in(0,1)$ by
\begin{align}\label{e:Pbeta}
    \Pbeta g := \sum_{j=1}^{N_h} \lambda_j^\beta \scalar{g,e_j}{H} \, e_{j,h}.
\end{align}

Since $V_h$ is finite-dimensional, the operator $Q_{h,k}^\beta \from V_h \to V_h$ in~\eqref{e:def:Qhk}
is bounded, $Q_{h,k}^\beta \in \cL(V_h)$ for short, with norm
\begin{align*}
    \norm{Q_{h,k}^\beta}{\cL(V_h)} := \sup_{\psi_h\in V_h\setminus\{0\}} \frac{\norm{Q_{h,k}^\beta \psi_h}{H}}{\norm{\psi_h}{H}} < \infty.
\end{align*}

We now consider the following stochastic partial differential equation
\begin{align}
    \mathrm{d} \widetilde{Y}(t) = Q_{h,k}^\beta \Pbeta \, \mathrm{d} W^{\beta}(t),
    \quad
    t \in [0,1],
    \qquad
    \widetilde{Y}(0) = Q_{h,k}^\beta \proj g. \label{e:spde-Ytilde}
\end{align}
Note that the reproducing kernel Hilbert space
of $W^\beta$ is $\Hdot{2\beta}$.
The finite rank of
the operator
$Q_{h,k}^{\beta} \Pbeta \from H \to V_h$ 
implies that it
is a Hilbert--Schmidt operator
from $\Hdot{2\beta}$ to $H$.
For this reason, existence
and uniqueness of a (strong) solution
$\widetilde{Y}$ to~\eqref{e:spde-Ytilde}
is evident.
Furthermore,
the solution process
$\widetilde{Y}$ satisfies
\begin{align*}
    \widetilde{Y}(1) &= \widetilde{Y}(0) + \int_0^1 Q_{h,k}^\beta \Pbeta \rd W^{\beta}(t)
        = Q_{h,k}^\beta (\proj g + \white_h^\cE),
\end{align*}
where $\white_h^\cE := \sum_{j=1}^{N_h} B_j(1) \, e_{j,h}$.
To see that also
$\widetilde{Y}(1) \overset{d}{=} u_{h,k}^Q$ holds in $H$,
define the deterministic matrix $\Rmat$
and the random vector $\Bmat_1$ by
\begin{align*}
    R_{ij} := \scalar{e_{i,h}, \phi_{j,h}}{H},
    \quad
    1\leq i,j \leq N_h,
    \qquad
    \Bmat_1 := (B_1(1), \ldots, B_{N_h}(1))^\T,
\end{align*}
i.e., $\Bmat_1$ is the vector
of the first $N_h$ Brownian motions at time $t=1$.
Due to
\begin{align*}
    (\Rmat^\T \Rmat)_{ij} = \scalar{\phi_{i,h}, \phi_{j,h}}{H} = M_{ij},
\end{align*}
the vector $\xivec := \Rmat^{-1} \Bmat_1$ is $\cN(\mathbf{0},\Mmat^{-1})$-distributed.
In addition, by~\cite[Lem.~2.8]{bkk17} the $V_h$-valued random variables
\begin{align*}
    \white_h^\cE = \sum_{j=1}^{N_h} B_j(1) \, e_{j,h}
    \quad
    \text{and}
    \quad
    \white_h^\Phi := \sum_{j=1}^{N_h} \xi_j \, \phi_{j,h}
\end{align*}
are equal in $L_2(\Omega; H)$.
In particular, their first and second moments coincide.
Since $\white_h^\cE$ and $\white_h^\Phi$
are Gaussian random variables,
their distributions are uniquely characterized
by their first two moments and we conclude that
\begin{align}\label{e:uQ-Ytilde}
    \widetilde{Y}(1)
    =
    Q_{h,k}^\beta (\proj g + \white_h^\cE )
    \overset{d}{=}
    Q_{h,k}^\beta (\proj g + \white_h^\Phi)
    =
    u_{h,k}^Q.
\end{align}

\subsection{The Kolmogorov backward equation and partition of the error}\label{subsec:kolmo}

With the aim of bounding the weak error in~\eqref{e:def:err-weak}
by means of It\^{o} calculus, we introduce
the following Kolmogorov backward equation associated with
the stochastic partial differential equation~\eqref{e:spde-Y} for $Y$
and the function $\varphi$ by
\begin{align}\label{e:def:kolmo}
    w_t(t,x) + \frac{1}{2} \tr\left( w_{xx}(t,x) L^{-2\beta} \right) = 0,
    \quad
    t\in[0,1], \
    x\in H,
    \qquad
    w(1,x) = \varphi(x).
\end{align}
Here, $w_x := D_x w$ and $w_{xx} := D^2_x w$ denote
the first and second order Fr\'{e}chet derivative of $w$
with respect to $x \in H$.
It is well-known \cite[Rem.~3.2.1, Thm.~3.2.3]{daprato2002} that
the solution $w \from [0,1] \times H \to \bbR$ to~\eqref{e:def:kolmo}
is given in terms of the stochastic process $Y$ in~\eqref{e:spde-Y}
by the following expectation
\begin{align}\label{e:w}
    w(t, x) = \bbE[ \varphi(x + Y(1) - Y(t)) ].
\end{align}
Since $\varphi\from H \to \bbR$ is twice continuously Fr\'{e}chet differentiable,
we can furthermore express the first two derivatives of $w$ with respect to $x$
in terms of $\varphi$ and $Y$ by
\begin{align}
    w_x(t,x) &= \bbE[ D\varphi(x + Y(1) - Y(t)) ], \label{e:wx} \\
    w_{xx}(t,x) &= \bbE[ D^2 \varphi(x + Y(1) - Y(t)) ]. \label{e:wxx}
\end{align}

Let $\widetilde{Y}$ be the solution to~\eqref{e:spde-Ytilde}.
The application of It\^{o}'s lemma \cite{brzezniak2003}
to the stochastic process
$(w(t,\widetilde{Y}(t)), \ t\in[0,1])$ yields
\begin{align}
\mathrm{d} w(t,\widetilde{Y}(t))
    &= \left( w_t( t, \widetilde{Y}(t) ) + \frac{1}{2} \tr \left( w_{xx}( t, \widetilde{Y}(t) ) Q_{h,k}^{\beta} \Pbeta L^{-2\beta} \bigl( Q_{h,k}^{\beta} \Pbeta\bigr)^* \right) \right) \rd t \notag \\
    &\qquad + w_x( t, \widetilde{Y}(t) ) Q_{h,k}^\beta \Pbeta \rd W^\beta(t),
     \qquad t\in[0,1], \label{e:Ito-lemma}
\end{align}
where, for $T\in\cL(H)$, the $H$-adjoint operator is denoted by $T^*$.
To simplify the second term in~\eqref{e:Ito-lemma},
we define the operator $\ptilde\from H \to V_h$ by
\begin{align}\label{e:def:ptilde}
    \ptilde g := \sum_{j=1}^{N_h} \scalar{g,e_j}{H} \, e_{j,h}.
\end{align}
Note that in contrast to the $H$-orthogonal projection $\proj$,
the operator $\ptilde$ is neither self-adjoint ($\ptilde^* \neq \ptilde$)
nor a projection ($\ptilde^2 \neq \ptilde$).
We then use the following relation between $\ptilde$ and
$\Pbeta$ from~\eqref{e:Pbeta},
\begin{align*}
\Pbeta L^{-\beta} g = \ptilde g
\qquad
\forall g\in H,
\end{align*}
and express~\eqref{e:Ito-lemma} as an integral
equation for $t=1$.
Taking the expectation on both sides of this equation
yields
\begin{align}
\begin{split}
\bbE[ w( 1, \widetilde{Y}(1) ) ]
    &= w( 0, Q_{h,k}^\beta \proj g ) \\
    &\quad + \frac{1}{2} \,
        \bbE \int_0^1 \tr \left( w_{xx}( t, \widetilde{Y}(t) ) \left(Q_{h,k}^{\beta} \ptilde {\ptilde}^*  Q_{h,k}^{\beta *} - L^{-2\beta} \right) \right) \rd t
        \label{e:w-Ytilde}
\end{split}
\end{align}
since
$\widetilde{Y}(0) = Q_{h,k}^\beta \proj g$ by \eqref{e:spde-Ytilde}
and
$w_t( t,\widetilde{Y}(t) ) = - \tfrac{1}{2}  \tr \bigl( w_{xx}( t, \widetilde{Y}(t) ) L^{-2\beta} \bigr)$
by~\eqref{e:def:kolmo}.

As a final step in this subsection,
we relate the quantity of interest $\bbE[\varphi(u)]$
with the expected value of $w(1,Y(1))$
and similarly for the approximation $\bbE[\varphi(u_{h,k}^Q)]$
and $w(1,\widetilde{Y}(1))$.
For this purpose, we extend the equalities 
in~\eqref{e:w}--\eqref{e:wxx}
to the case that $x=\xi$ is a an $H$-valued random variable
in the following lemma.

\begin{lemma}\label{l:w-phi}
Let Assumption~\ref{ass:all}\ref{ass:phi} be satisfied.
Then, for every $t\in[0,1]$ and any $\cF_t$-measurable random variable $\xi\in L_{p+2}(\Omega; H)$, it holds
\begin{align*}
    D_x^k w( t, \xi ) = \bbE[ D^k \varphi( \xi + Y(1) - Y(t) ) \, | \, \cF_t],
    \quad
    k\in\{0,1,2\}.
\end{align*}
\end{lemma}

\begin{proof}
For $k=0$, this identity follows from~\cite[Lem.~4.1]{kovacs2014}
with $N=p+2$, $\xi_1 = \xi$ and $\xi_2 = Y(1) - Y(t)$,
since $Y(t) \in L_{p+2}(\Omega; H)$
for all $t\in[0,1]$
by Lemma~\ref{l:Ymoments} and
$|\varphi(x)| \lesssim 1 + \norm{x}{H}^{p+2}$ as
a consequence of~\eqref{e:ass:phi-poly}.

Furthermore, for $y,\,z\in H$,
we define $\varphi_{y}, \varphi_{y,\,z} \from H \to \bbR$ by
\begin{align*}
    \varphi_{y}   (x) := \scalar{ D\varphi(x), y}{H}, \qquad
    \varphi_{y,\,z} (x) := \scalar{ D^2 \varphi(x) z, y}{H}.
\end{align*}
Since the inner product is bilinear and continuous with respect to both components,
we find with~\eqref{e:wx}--\eqref{e:wxx} that
\begin{align*}
    \scalar{ w_x(t,x), y}{H} &= \bbE[ \varphi_{y}( x + Y(1) - Y(t) ) ] , \\
    \scalar{ w_{xx}(t,x) z, y}{H} &= \bbE[ \varphi_{y,\,z}( x + Y(1) - Y(t) ) ].
\end{align*}
Thus, again applying~\cite[Lem.~4.1]{kovacs2014}
for $\xi_1 = \xi$ and $\xi_2 = Y(1) - Y(t)$
as well as $N=p+1$ and $N=p$, respectively,
yields
\begin{align*}
    \scalar{ w_x(t,\xi), y}{H} &= \bbE[ \varphi_{y}( \xi_1 + \xi_2) \, | \, \cF_t]
        = \scalar{ \bbE[ D\varphi( \xi + Y(1) - Y(t) ) \, | \, \cF_t], y}{H}, \\
    \scalar{ w_{xx}(t,\xi) z, y}{H} &= \bbE[ \varphi_{y,\,z}( \xi_1 + \xi_2 ) \, | \, \cF_t]
        = \scalar{ \bbE[ D^2\varphi( \xi + Y(1) - Y(t) ) \, | \, \cF_t] z, y}{H}
\end{align*}
by bilinearity and continuity of the inner product.
The separability of $H$
and the arbitrary choice of $y,\,z\in H$ complete
the proof of the assertion for $k\in\{1,2\}$.
\end{proof}

Owing to Lemma~\ref{l:w-phi} 
and the tower property for conditional expectations,
the stochastic process $(w(t,Y(t)), \ t\in [0,1])$
has no drift, i.e.,
\begin{align}\label{e:Ew-Y0}
    \bbE[ w( 1, Y(1) )]
    = \bbE[ \varphi(Y(1)) ]
    = \bbE[ w( 0, Y(0) )]
    = w( 0, L^{-\beta}g ).
\end{align}
Furthermore, it follows with~\eqref{e:u-Y} and~\eqref{e:uQ-Ytilde} that
\begin{align}\label{e:Ew-Y1}
    \bbE[ w( 1, Y(1) )]
        &= \bbE[ \varphi( Y(1) ) ]
         = \bbE[ \varphi(u) ], \\
    \bbE[ w( 1, \widetilde{Y}(1) ) ]
        &= \bbE[ \varphi( \widetilde{Y}(1) ) ]
         = \bbE[ \varphi( u_{h,k}^Q ) ]. \label{e:Ew-Ytilde}
\end{align}

Summing up the observations
in~\eqref{e:w-Ytilde}--\eqref{e:Ew-Ytilde},
we find that
the difference between
the quantity of interest $\bbE[\varphi(u)]$
and the expected value of the approximation~$\varphi(u_{h,k}^Q)$
can be expressed by
\begin{align*}
    \bbE[ \varphi(u) ] - \bbE[ \varphi( u_{h,k}^Q ) ]
        &= w( 0, L^{-\beta} g ) - w( 0, Q^\beta_{h,k} \proj g )  \\
        &\quad  - \frac{1}{2} \, \bbE \int_0^1 \tr \left( w_{xx}( t, \widetilde{Y}(t) ) \left( Q_{h,k}^{\beta} \ptilde \ptilde^*  Q_{h,k}^{\beta *} - L^{-2\beta} \right) \right) \rd t.
\end{align*}
This equality implies
that the weak error~\eqref{e:def:err-weak}
admits the following upper bound
\begin{align}
\begin{split}
\bigl| \bbE[ \varphi(u) ] &- \bbE[ \varphi( u_{h,k}^Q ) ] \bigr|
    \leq \bigl| w( 0, L^{-\beta} g ) - w( 0, L_h^{-\beta} \proj g ) \bigr| \\
        &\quad + \bigl| w( 0, L_h^{-\beta} \proj g) - w( 0, Q^\beta_{h,k} \proj g ) \bigr| \\
        &\quad + \frac{1}{2} \biggl| \bbE \int_0^1 \tr \left( w_{xx}( t, \widetilde{Y}(t) )
                                \left( \Qtilde \QtildeT - \Lhtilde \LhtildeT \right) \right) \rd t \biggr| \\
        &\quad + \frac{1}{2} \biggl| \bbE \int_0^1 \tr \left( w_{xx}( t, \widetilde{Y}(t) )
                                \left(\Lhtilde \LhtildeT - L^{-2\beta} \right) \right) \rd t \biggr| \\
        &=: \text{(I)} + \text{(II)} + \text{(III)} + \text{(IV)}, \label{e:err-partition}
\end{split}
\end{align}
where we set $\Qtilde :=  Q_{h,k}^{\beta} \ptilde$ and
$\Lhtilde := L_h^{-\beta} \ptilde$.

The following subsections are structured as follows:
In \S\ref{subsec:err-det} we bound the deterministic
error $\norm{(L^{-\beta} - L_h^{-\beta} \proj)g}{H}$
caused by the finite element discretization.
This result is essential for estimating
the first error term (I) in~\eqref{e:err-partition}.
Secondly, we investigate the terms (II) and (III)
stemming from applying the quadrature
operator $Q_{h,k}^\beta$ instead of the discrete
fractional inverse $L_h^{-\beta}$ in \S\ref{subsec:err-quad}.
Finally, in \S\ref{subsec:proof} we
estimate the trace in (IV) and combine
all our results to prove Theorem~\ref{thm:weak-conv}.

\subsection{The deterministic finite element error}\label{subsec:err-det}

In this subsection we focus on the deterministic error
$\norm{(L^{-\beta} - L_h^{-\beta} \proj)g}{H}$
caused by the inhomogeneity $g$.
More precisely,
we derive an explicit rate of convergence
depending on the $\Hdot{\theta}$-regularity of $g$
in Lemma~\ref{l:err-det} below.
Subsequently, in Lemma~\ref{l:err-I},
we apply this result
to bound the first term
of~\eqref{e:err-partition}.

\begin{lemma}\label{l:err-det}
Suppose Assumption~\ref{ass:all}\ref{ass:Sh} is satisfied.
Set $\theta_{*} := d(2\alpha\beta - 1) - 2\beta$ and let
$\theta > \min\{ \theta_{*}, s - 2\beta \}$ if $\theta_{*} \geq 0$,
and set $\theta = 0$ otherwise.
Then there exists a constant $C>0$, independent of $h$, such that
\begin{align}\label{e:err-det}
    \norm{(L^{-\beta} - L_h^{-\beta} \proj)g}{H} \leq C h^{\min\{d(2\alpha\beta-1), s\}} \norm{g}{\theta}
\end{align}
for all $g\in \Hdot{\theta}$ and sufficiently small $h\in(0,h_0)$.
\end{lemma}

\begin{proof}
By applying \cite[Ch.~2, Eq.~(6.9)]{pazy1983} to the negative fractional powers of $L$ and~$L_h$, we find
\begin{align*}
    L^{-\beta} - L_h^{-\beta} \proj =
        \frac{1}{\Gamma(\beta)} \int_0^\infty t^{\beta-1} (S(t) - S_h(t) \proj) \rd t .
\end{align*}
Thus, Assumption~\ref{ass:all}\ref{ass:Sh} yields for
$0 \leq \theta_j \leq \sigma_j \leq s$ ($j=1,2$)
the estimate
\begin{align*}
    \norm{(L^{-\beta} - L_h^{-\beta} \proj)g}{H}
        &\lesssim h^{\sigma_1} \norm{g}{\theta_1} \int_0^1 t^{\beta-1+ \frac{\theta_1 - \sigma_1}{2}} \rd t
            + h^{\sigma_2} \norm{g}{\theta_2} \int_1^\infty t^{\beta-1+ \frac{\theta_2 - \sigma_2}{2}} \rd t .
\end{align*}

If $\theta_{*} \geq 0$, we let
$\epsilon>0$ be such that $\theta = \min\{\theta_{*}, s-2\beta\} + \epsilon$
and we choose
$\sigma_1 := \min\{d(2\alpha\beta-1),s\}$,
$\sigma_2 := s$,
$\theta_1 := \min\{ \theta, \sigma_1 \}$, and
$\theta_2 := 0$.
We then obtain
$\theta_1 - \sigma_1 = \min\{-2\beta + \epsilon, 0\}$
and
\begin{align*}
    \norm{(L^{-\beta} - L_h^{-\beta} \proj)g}{H}
        &\lesssim h^{\min\{d(2\alpha\beta-1),s\}}
            \Bigl( \tfrac{2}{\min\{\epsilon, 2\beta\}} \norm{g}{\theta_1} + \tfrac{2}{s-2\beta} \norm{g}{H} \Bigr).
\end{align*}

For $\theta_{*} < 0$, we instead set
$\sigma_1 := d(2\alpha\beta-1)$,
$\sigma_2 := s$,
$\theta_1 := 0$,
$\theta_2 := 0$,
and we conclude in a similar way that
\begin{align*}
    \norm{(L^{-\beta} - L_h^{-\beta} \proj)g}{H}
        &\lesssim h^{\min\{d(2\alpha\beta-1),s\}} \norm{g}{H} ( - 2 \theta_{*}^{-1} + 2 (s-2\beta)^{-1} ).
\end{align*}

Since in both cases
$\max\{ \norm{g}{\theta_1}, \norm{g}{\theta_2} \} \leq \norm{g}{\theta}$
with $\theta$ defined as in the statement of the lemma,
the bound~\eqref{e:err-det} follows.
\end{proof}

\begin{remark}
We note that by letting
$\sigma_1 = \sigma_2 := s$,
$\theta_1 := s-2\beta+\epsilon$, and
$\theta_2 := 0$ in the proof of Lemma~\ref{l:err-det}
the optimal convergence rate for the deterministic error,
\[
    \norm{(L^{-\beta} - L_h^{-\beta} \proj)g}{H} \leq C h^s \norm{g}{s-2\beta+\epsilon},
\]
can be derived.
The error estimate~\eqref{e:err-det}
is formulated in such a way that
the smoothness 
of $g\in\Hdot{\theta}$
is minimal for convergence with the rate
$\min\{d(2\alpha\beta-1), s\}$, which will
dominate the overall weak error,
stemming from the term (IV) in the partition~\eqref{e:err-partition},
see Lemma~\ref{l:err-IV}.

We furthermore remark that the convergence
result of Lemma~\ref{l:err-det}
is in accordance with the result of~\cite[Thm.~4.3]{bonito2015}.
There the self-adjoint positive definite operator $L$ is induced by
an $H_0^1(\cD)$-coercive, symmetric bilinear form $A$:
\begin{align*}
    \duality{L v, w}{} := A(v,w) = \int_\cD a(\xvec) \nabla v(\xvec) \cdot \nabla w(\xvec) \rd \xvec
    \quad
    \forall v,w \in \Hdot{1},
\end{align*}
where $0 < a_0 \leq a(\xvec) \leq a_1$,
$H := L_2(\cD)$, $\Hdot{1} := H^1_0(\cD)$ and
$\cD\subset\bbR^d$, $d\in\{1,2,3\}$, is a bounded polygonal domain
with Lipschitz boundary.
The discrete spaces $(V_h)_h$ considered in~\cite{bonito2015}
are the finite element spaces with continuous piecewise linear basis functions
defined with respect to a quasi-uniform family of triangulations.
The convergence rate for the error $\norm{(L^{-\beta} - L_h^{-\beta}\proj)g}{H}$
derived in~\cite[Thm.~4.3]{bonito2015} is
$2 \tau$, if $g\in \Hdot{\theta}$
for $\theta > 2 (\tau - \beta)$, if $\tau \geq \beta$,
and $\theta = 0$ otherwise.
Here, $\tau\in(0,1]$ is
such that the operators
\begin{align*}
    L^{-1} \from \widetilde{H}^{-1+\tau}(\cD) \to \widetilde{H}^{1+\tau}(\cD)
    \quad
    \text{and}
    \quad
    L \from \widetilde{H}^{1+\tau}(\cD) \to \widetilde{H}^{-1+\tau}(\cD)
\end{align*}
are bounded with respect to the intermediate Sobolev spaces
\begin{align*}
    \widetilde{H}^\varrho(\cD)
        &:= \begin{cases}
            H_0^1(\cD) \cap H^\varrho(\cD),         & \varrho \in [1, 2], \\
            [ L_2(\cD), H_0^1(\cD) ]_{\varrho,2},   & \varrho \in [0, 1], \\
            [ H^{-1}(\cD), L_2(\cD)]_{1+\varrho,2}, & \varrho \in [-1,0],
        \end{cases}
\end{align*}
where $H^{-1}(\cD) = \Hdot{-1}$ is the dual space
of $H^1_0(\cD) = \Hdot{1}$ and
$[\cdot, \cdot]_{\varrho,q}$ denotes
the real $K$-interpolation method.

According to this result of~\cite{bonito2015},
the convergence rate
$2\min\{d(\alpha\beta-1/2), 1\}$
can be achieved
if $g$ is $\Hdot{\theta}$-regular
for $\theta > \theta_*$ if
$\theta_* := 2 (\min\{d(\alpha\beta-1/2),1\} - \beta) \geq 0$
and $\theta = 0$ if $\theta_* < 0$.
A comparison with~\eqref{e:err-det} in Lemma~\ref{l:err-det}
shows that the error estimates and regularity assumptions coincide
for this particular case,
since $s=2$ for the choice of finite-dimensional subspaces $(V_h)_h$
in~\cite{bonito2015} specified above.
\end{remark}

Having bounded the error
between $L^{-\beta}g$ and $L_h^{-\beta}\proj g$,
an estimate of the first error term (I) in~\eqref{e:err-partition}
is an immediate consequence of the fundamental theorem
of calculus and the chain rule for Fr\'{e}chet derivatives.
This bound is formulated in the next lemma.

\begin{lemma}\label{l:err-I}
Let Assumptions~\ref{ass:all}\ref{ass:Sh}--\ref{ass:phi}
be satisfied and $2\alpha\beta>1$.
Define $\theta\geq 0$ as in Lemma~\ref{l:err-det}.
Then there exists a constant $C>0$, independent of $h$, such that
\begin{align*}
    \bigl| w( 0, L^{-\beta} g ) - w( 0, L_h^{-\beta} \proj g ) \bigr|
        \leq C h^{\min\{d(2\alpha\beta-1), s\}} \norm{g}{\theta} ( 1 + \norm{g}{H}^{p+1} )
\end{align*}
for all $g\in \Hdot{\theta}$ and sufficiently small $h\in(0,h_0)$.
\end{lemma}

\begin{proof}
Since the mapping $x \mapsto w(0,x)$ is Fr\'{e}chet differentiable,
we obtain by the fundamental theorem of calculus and the Cauchy--Schwarz inequality
\begin{align*}
    \bigl| w( 0, &\, L_h^{-\beta} \proj g ) - w( 0, L^{-\beta} g ) \bigr| \\
        &= \Bigl| \int_0^1 \scalar{ w_x(0, L^{-\beta} g + t ( L_h^{-\beta} \proj - L^{-\beta}) g ), (L_h^{-\beta} \proj - L^{-\beta} )g }{H} \rd t \Bigr| \\
        &\leq \norm{(L_h^{-\beta} \proj - L^{-\beta} )g }{H} \sup_{t\in[0,1]} \norm{ w_x(0, L^{-\beta} g + t ( L_h^{-\beta} \proj - L^{-\beta}) g ) }{H}.
\end{align*}
A bound for the first term is given by~\eqref{e:err-det} in Lemma~\ref{l:err-det}.
For the second term, we use~\eqref{e:wx}, $Y(0) = L^{-\beta}g$,
and the polynomial growth~\eqref{e:ass:phi-poly}
of $D^2\varphi$ to estimate
\begin{align*}
\norm{ w_x(0, L^{-\beta} g + t ( L_h^{-\beta} \proj - L^{-\beta}) g ) }{H}
    &\leq \bbE[ \norm{ D\varphi( Y(1) + t ( L_h^{-\beta} \proj - L^{-\beta}) g )}{H} ] \\
    &\lesssim ( 1 + \bbE[ \norm{ Y(1)}{H}^{p+1} ] + \norm{g}{H}^{p+1} )
\end{align*}
for all $t\in[0,1]$.
The boundedness~\eqref{e:Lp-Y} of the $(p+1)$-th moment of $Y(1)$ completes the proof,
since the trace of $L^{-2\beta}$ is finite if $2\alpha\beta > 1$.
\end{proof}

\subsection{The quadrature approximation}\label{subsec:err-quad}

In this subsection we address the error terms (II) and (III)
in~\eqref{e:err-partition}, which are induced by the quadrature
approximation $Q_{h,k}^\beta$ of $L_h^{-\beta}$.
To this end, we start by stating the following result
of~\cite[Lem.~3.4, Thm.~3.5]{bonito2015} that bounds the error
between the two operators on $V_h$.

\begin{lemma}\label{l:Qhk}
The approximation $Q_{h,k}^\beta \from V_h \to V_h$
of $L_h^{-\beta}$ in~\eqref{e:def:Qhk} admits the bound
\begin{align*}
    \norm{( Q_{h,k}^\beta - L_h^{-\beta} ) \phi_h}{H}
        &\leq C e^{-\frac{\pi^2}{2k}} \norm{\phi_h}{H}
        \quad
        \forall \phi_h \in V_h,
\end{align*}
and it is bounded,
$\norm{Q_{h,k}}{\cL(V_h)} \leq C'$,
for sufficiently small $h\in(0,h_0)$, $k\in(0,k_0)$,
where the constants $C,C'>0$ depend only on
$\beta$ and the smallest eigenvalue 
of $L$.
\end{lemma}

In the following, we use this error estimate of the quadrature
approximation $Q_{h,k}^\beta$ for bounding
the second term of~\eqref{e:err-partition} in Lemma~\ref{l:err-II}
as well as the trace occurring
in the third term of~\eqref{e:err-partition} in Lemma~\ref{l:err-III}.

\begin{lemma}\label{l:err-II}
Suppose that Assumption~\ref{ass:all}\ref{ass:phi} is satisfied and that $2\alpha\beta > 1$.
Then there exists a constant $C>0$, independent of $h$ and $k$, such that
\begin{align*}
    \bigl| w( 0, L_h^{-\beta} \proj g ) - w( 0, Q_{h,k}^\beta \proj g ) \bigr|
        \leq C e^{-\frac{\pi^2}{2k}} \norm{g}{H} ( 1 + \norm{g}{H}^{p+1} )
\end{align*}
for all $g\in H$ and sufficiently small $h\in(0,h_0)$ and $k\in(0,k_0)$.
\end{lemma}

\begin{proof}
As in the proof of Lemma~\ref{l:err-I},
we apply the fundamental theorem of calculus and the chain rule for Fr\'echet derivatives.
By~\eqref{e:wx} and Lemma~\ref{l:Qhk} we then find
\begin{align*}
    \bigl| w( 0, Q_{h,k}^\beta \proj g &) - w( 0, L_h^{-\beta} \proj g ) \bigr|
         \leq \norm{( Q_{h,k}^\beta - L_h^{-\beta} ) \proj g}{H} \\
        &\quad \times \sup_{t\in[0,1]} \bbE[ \norm{ D\varphi( L_h^{-\beta}\proj g + t ( Q_{h,k}^\beta - L_h^{-\beta} ) \proj g + Y(1) - L^{-\beta}g ) }{H} ] \\
        &\lesssim e^{-\frac{\pi^2}{2k}} \norm{g}{H} ( 1 + \bbE[ \norm{Y(1)}{H}^{p+1} ] + \norm{g}{H}^{p+1} ).
\end{align*}
Again, the proof is completed by~\eqref{e:Lp-Y} and the fact that $\tr(L^{-2\beta}) < \infty$.
\end{proof}

\begin{lemma}\label{l:err-III}
Let Assumptions~\ref{ass:all}\ref{ass:Vh-1}--\ref{ass:L} be satisfied.
Then there exists a constant $C>0$, independent of $h$ and $k$, such that
\begin{align*}
    \bigl| \tr( T ( \Qtilde \QtildeT - \Lhtilde \LhtildeT ) ) \bigr|
         \leq C \left( e^{-\frac{\pi^2}{k}} h^{-d} + e^{-\frac{\pi^2}{2k}} + e^{-\frac{\pi^2}{2k}} f_{\alpha,\beta}(h) \right) \norm{T}{\cL(H)}
\end{align*}
for every self-adjoint $T \in \cL(H)$ and sufficiently small $h\in(0,h_0)$ and $k\in(0,k_0)$.
Here, the function $f_{\alpha,\beta}$ is defined as in Theorem~\ref{thm:weak-conv}.
\end{lemma}

\begin{proof}
By the definition of $\ptilde$ in~\eqref{e:def:ptilde}
we have
\begin{align}\label{e:ptilde-e}
    \ptilde e_j = e_{j, h},
    \quad
    j\in\{1,\ldots,N_h\},
    \qquad
    \ptilde e_j = 0,
    \quad
    j > N_h.
\end{align}
Therefore, the trace of interest simplifies to a finite sum,
\begin{align}
    \tr( T ( \Qtilde \QtildeT - \Lhtilde \LhtildeT ) )
        &= \sum_{j=1}^{N_h} \bigl[ \scalar{T Q_{h,k}^\beta e_{j,h}, Q_{h,k}^\beta e_{j,h} }{H}
                -  \scalar{T L_h^{-\beta} e_{j,h}, L_h^{-\beta} e_{j,h} }{H} \bigr] \notag \\
        &= \sum_{j=1}^{N_h} \scalar{T (Q_{h,k}^\beta - L_h^{-\beta}) e_{j,h}, (Q_{h,k}^\beta - L_h^{-\beta}) e_{j,h} }{H} \notag \\
        &\qquad + 2 \sum_{j=1}^{N_h} \scalar{T (Q_{h,k}^\beta - L_h^{-\beta}) e_{j,h}, L_h^{-\beta} e_{j,h} }{H} \notag \\
        &=: S_1 + 2 S_2,
        \label{e:trTQ-Lhtilde}
\end{align}
where the second equality follows from the self-adjointness of $T\in\cL(H)$.

The application of the Cauchy--Schwarz inequality and
of Lemma~\ref{l:Qhk} to the first sum yield the following upper bound
\begin{align*}
    | S_1 |
        &\leq \norm{T}{\cL(H)} \sum_{j=1}^{N_h} \norm{(Q_{h,k}^\beta - L_h^{-\beta})e_{j,h}}{H}^2
         \leq C e^{-\frac{\pi^2}{k}} N_h \norm{T}{\cL(H)}.
\end{align*}
By Assumption~\ref{ass:all}\ref{ass:Vh-1} we thus have
$| S_1 |\lesssim e^{-\frac{\pi^2}{k}} h^{-d} \norm{T}{\cL(H)}$.

The second sum can be bounded by
\begin{align*}
    | S_2 |
        &\leq \norm{T}{\cL(H)} \max_{1\leq j \leq N_h}  \norm{(Q_{h,k}^\beta - L_h^{-\beta})e_{j,h}}{H}
            \sum_{j=1}^{N_h} \lambda_{j,h}^{-\beta}.
\end{align*}
Finally, due to the approximation property
of the discrete eigenvalues $\lambda_{j,h}$
in Assumption~\ref{ass:all}\ref{ass:Vh-2}
as well as the growth~\eqref{e:ass:lambdaj} of the
exact eigenvalues $\lambda_j$ we obtain
$\lambda_{j,h}^{-\beta} \leq \lambda_{j}^{-\beta} \leq c_{\lambda}^{-\beta} j^{-\alpha\beta}$
and, for $\alpha\beta\neq 1$, we find
\begin{align*}
    | S_2 |
        &\lesssim e^{-\frac{\pi^2}{2k}} (1 + N_h^{1-\alpha\beta} ) \norm{T}{\cL(H)}
        \lesssim e^{-\frac{\pi^2}{2k}} (1 + h^{d(\alpha\beta - 1)} ) \norm{T}{\cL(H)},
\end{align*}
where we have used
Lemma~\ref{l:Qhk} and Assumption~\ref{ass:all}\ref{ass:Vh-1}.
If $\alpha\beta=1$, we instead estimate
$| S_2 | \lesssim e^{- \pi^2 / (2k)} (1 + |\ln(h)| ) \, \norm{T}{\cL(H)}$.
This completes the proof.
\end{proof}

\subsection{Proof of Theorem~\ref{thm:weak-conv}}\label{subsec:proof}

After having bounded the terms (I), (II), and (III)
in the partition~\eqref{e:err-partition} of the weak error
in the previous subsections,
we now turn to estimating the final error term (IV).
Furthermore, 
we bound the $p$-th moment of $\widetilde{Y}(t)$,
where $\widetilde{Y}$ is the solution process
of~\eqref{e:spde-Ytilde}.
We then combine all our results
and prove Theorem~\ref{thm:weak-conv}.

\begin{lemma}\label{l:err-IV}
Let Assumptions~\ref{ass:all}\ref{ass:Vh-1}--\ref{ass:L} be satisfied.
Then there exists a constant $C>0$, independent of $h$, such that
\begin{align*}
    \bigl| \tr\bigl( T \bigl( \Lhtilde \LhtildeT - L^{-2\beta} \bigr) \bigr) \bigr|
        \leq C h^{\min\{d(2\alpha\beta-1),r,s\}} \norm{T}{\cL(H)}
\end{align*}
for every self-adjoint $T \in \cL(H)$ and sufficiently small $h\in(0,h_0)$.
\end{lemma}

\begin{proof}
Similarly to~\eqref{e:trTQ-Lhtilde}
we use the self-adjointness of $T$ and rewrite the trace as
$\tr(T(\Lhtilde \LhtildeT - L^{-2\beta})) = S_1 + S_2$,
where
\begin{align*}
    S_1 := \sum_{j\in\bbN} \scalar{T(\Lhtilde - L^{-\beta})e_j, \Lhtilde e_j}{H},
    \qquad
    S_2 := \sum_{j\in\bbN} \scalar{T(\Lhtilde - L^{-\beta})e_j, L^{-\beta} e_j}{H}.
\end{align*}
In order to estimate the terms $S_1$ and $S_2$, we
note that for $j\in\{1,\ldots,N_h\}$
\begin{align*}
    \norm{(\Lhtilde - L^{-\beta}) e_j}{H}
        &= \norm{ \lambda_{j,h}^{-\beta} e_{j,h} - \lambda_j^{-\beta} e_j}{H}
         \leq | \lambda_{j,h}^{-\beta} - \lambda_j^{-\beta} | + \lambda_j^{-\beta} \norm{e_{j,h} - e_j}{H}.
\end{align*}
By the mean value theorem, the existence of $\widetilde{\lambda}_j\in(\lambda_j,\lambda_{j,h})$ satisfying
$\lambda_{j}^{-\beta} - \lambda_{j,h}^{-\beta} = \beta \widetilde{\lambda}_j^{-\beta-1} (\lambda_{j,h} - \lambda_j)$
is ensured.
By Assumption~\ref{ass:all}\ref{ass:Vh-2} we thus have
\begin{align}\label{e:LhtildeLbeta}
    \norm{(\Lhtilde - L^{-\beta}) e_j}{H}
        &\leq \max\bigl\{\beta C_1, \sqrt{C_2} \bigr\} \left( h^r \lambda_j^{q-\beta-1} +  h^s \lambda_j^{\frac{q}{2} - \beta} \right).
\end{align}

Owing to~\eqref{e:ptilde-e} the series $S_1$ simplifies to the finite sum
\begin{align*}
    S_1 = \sum_{j=1}^{N_h} \lambda_{j,h}^{-\beta} \, \scalar{T(\Lhtilde - L^{-\beta}) e_j, e_{j,h}}{H}.
\end{align*}
Using~\eqref{e:LhtildeLbeta}
as well as Assumptions~\ref{ass:all}\ref{ass:Vh-1}--\ref{ass:L},
this sum can be bounded by
\begin{align*}
    |S_1|
        &\lesssim \norm{T}{\cL(H)} \sum_{j=1}^{N_h} \left( h^r \lambda_j^{q-2\beta-1} + h^s  \lambda_j^{\frac{q}{2}-2\beta} \right)
         \lesssim h^{\min\{d(2\alpha\beta-1),r,s\}} \norm{T}{\cL(H)},
\end{align*}
since $d \alpha (q-1) \leq r$ and $d \alpha q/2 \leq s$ by~Assumption~\ref{ass:all}\ref{ass:L}.

For the second term we find
\begin{align*}
    S_2
        = \sum_{j = 1}^{N_h} \lambda_{j}^{-\beta} \scalar{T(\Lhtilde - L^{-\beta}) e_j, e_j}{H}
            - \sum_{j > N_h} \lambda_{j}^{-2\beta} \scalar{T e_j, e_j}{H},
\end{align*}
since $\Lhtilde e_j = 0$ for $j>N_h$ by~\eqref{e:ptilde-e}.
Therefore, the application of \eqref{e:LhtildeLbeta} yields
\begin{align*}
    |S_2|
        &\lesssim \norm{T}{\cL(H)}
            \Biggl( \sum_{j=1}^{N_h} \left( h^r \lambda_{j}^{q-2\beta-1} + h^s \lambda_{j}^{\frac{q}{2}-2\beta} \right)
                + \sum_{j>N_h} \lambda_{j}^{-2\beta} \Biggr)
\end{align*}
and $|S_2| \lesssim h^{\min\{d(2\alpha\beta-1),r,s\}} \norm{T}{\cL(H)}$ follows
from Assumptions~\ref{ass:all}\ref{ass:Vh-1},~\ref{ass:L}.
\end{proof}

\begin{lemma}\label{l:Ytildemoments}
Suppose that Assumptions~\ref{ass:all}\ref{ass:Vh-1}--\ref{ass:L} are satisfied.
Let $p\in\bbN$, $t\in[0,1]$, and $\widetilde{Y}$ be the strong solution
of~\eqref{e:spde-Ytilde}.
Then the $p$-th moment of $\widetilde{Y}(t)$ exists
and, for $p \geq 2$, it admits the following bound:
\begin{align*} 
    \bbE\bigl[ \norm{\widetilde{Y}(t)}{H}^p \bigr]
        &\leq C \Bigl( 1 + e^{-\frac{p \pi^2}{2k}} h^{-\frac{pd}{2}} + \norm{g}{H}^p \Bigr),
\end{align*}
where
the constant $C>0$ is independent of $h$ and $k$.
\end{lemma}

\begin{proof}
Since $P_h^\beta W^\beta(t) = \sum_{j=1}^{N_h} B_j(t) \, e_{j,h}$,
we obtain by Lemma~\ref{l:Qhk},
for any $p\geq 2$, that
\begin{align*}
    \bbE \bigl[ \norm{(Q_{h,k}^\beta - L_h^{-\beta}) P_h^\beta W^\beta(t)}{H}^p \bigr]
        &\leq C^p e^{-\frac{p \pi^2}{2k}} \, \bbE \Bigl| \sum_{j=1}^{N_h} B_j(t)^2 \Bigr|^{\frac{p}{2}}
         \leq C^p e^{-\frac{p \pi^2}{2k}} N_h^{\frac{p}{2}} t^{\frac{p}{2}} \mu_p,
\end{align*}
where, again,
$\mu_p := \bbE[ |Z|^p ]$
denotes the $p$-th absolute moment
of $Z\sim\cN(0,1)$ and
the constant $C>0$ is independent of $h$, $k$, and $p$.
Furthermore, using $0 < \lambda_{j} \leq \lambda_{j,h}$
of Assumption~\ref{ass:all}\ref{ass:Vh-2}
and applying the H\"{o}lder inequality
gives
\begin{align*}
    \bbE \bigl[ \norm{L_h^{-\beta} P_h^\beta W^\beta(t)}{H}^p \bigr]
        &= \bbE \Bigl| \sum_{j=1}^{N_h} \lambda_{j,h}^{-2\beta} B_j(t)^2 \Bigr|^{\frac{p}{2}}
        \leq \tr(L^{-2\beta})^{\frac{p}{2}} t^{\frac{p}{2}} \mu_p,
\end{align*}
where $\tr(L^{-2\beta}) < \infty$ by Assumption~\ref{ass:all}\ref{ass:L}.
Thus, we obtain for the solution $\widetilde{Y}$ of~\eqref{e:spde-Ytilde}
that for any $t\in[0,1]$ the bound
\begin{align*}
    \bbE\bigl[ &\norm{\widetilde{Y}(t)}{H}^p \bigr]
         = \bbE\bigl[ \norm{Q_{h,k}^\beta \proj g + (Q_{h,k}^\beta - L_h^{-\beta}) P_h^\beta W^\beta(t) + L_h^{-\beta} P_h^\beta W^\beta(t)}{H}^p \bigr] \\
        &\leq 3^{p-1} \Bigl( \norm{Q_{h,k}^\beta \proj g}{H}^p
                + \bbE\bigl[ \norm{(Q_{h,k}^\beta - L_h^{-\beta}) P_h^\beta W^\beta(t)}{H}^p \bigr]
                + \bbE\bigl[ \norm{L_h^{-\beta} P_h^\beta W^\beta(t)}{H}^p \bigr] \Bigr) \\
        &\leq 3^{p-1} \Bigl(\norm{Q_{h,k}^\beta}{\cL(V_h)}^p \norm{g}{H}^p
            + C^p e^{-\frac{p \pi^2}{2k}} N_h^{\frac{p}{2}} t^{\frac{p}{2}} \mu_p
            + \tr(L^{-2\beta})^{\frac{p}{2}} t^{\frac{p}{2}} \mu_p \Bigr)
\end{align*}
holds. Finally, the assertion follows by
the boundedness of $Q_{h,k}^\beta$ which is uniform in $h$ and $k$,
the finiteness of $\tr(L^{-2\beta})$, and
Assumption~\ref{ass:all}\ref{ass:Vh-1}.
\end{proof}

\begin{proof}[Proof of Theorem~\ref{thm:weak-conv}]
Owing to the partition~\eqref{e:err-partition}
and the estimates of the error terms (I)--(IV)
in Lemmata~\ref{l:err-I} and \ref{l:err-II}--\ref{l:err-IV}
we can bound the weak error as follows
\begin{align*}
    \bigl| \bbE[\varphi(u)] &- \bbE[\varphi(u_{h,k}^Q)] \bigr|
         \lesssim \left( h^{\min\{d(2\alpha\beta-1),s\}} + e^{-\frac{\pi^2}{2k}} \right) \norm{g}{\theta} (1 + \norm{g}{H}^{p+1}) \\
        &\quad + \sup_{t\in [0,1]} \bbE\bigl[\norm{w_{xx}(t,\widetilde{Y}(t))}{\cL(H)} \bigr]
                \left( e^{-\frac{\pi^2}{k}} h^{-d} + e^{-\frac{\pi^2}{2k}} +  e^{-\frac{\pi^2}{2k}} f_{\alpha,\beta}(h) \right) \\
        &\quad + \sup_{t\in [0,1]} \bbE\bigl[\norm{w_{xx}(t,\widetilde{Y}(t))}{\cL(H)} \bigr] \,
                h^{\min\{d(2\alpha\beta-1),r,s\}},
\end{align*}
since $w_{xx}(t,x)\in\cL(H)$ is self-adjoint for every $t\in[0,1]$ and $x\in H$.
The application of Lemma~\ref{l:w-phi} and of the tower property for conditional expectations yield
\begin{align*}
    \bbE[\norm{w_{xx}(t,\widetilde{Y}(t))}{\cL(H)}]
        &= \bbE[ \norm{ \bbE[ D^2 \varphi(\widetilde{Y}(t) + Y(1) - Y(t)) | \cF_t ] }{\cL(H)} ] \\
        &\leq \bbE[ \norm{ D^2 \varphi(\widetilde{Y}(t) + Y(1) - Y(t)) }{\cL(H)} ].
\end{align*}
By the polynomial growth~\eqref{e:ass:phi-poly}
of $D^2 \varphi$ and the boundedness of the $p$-th moments
of~$Y(t)$ and $\widetilde{Y}(t)$ in Lemmata~\ref{l:Ymoments}
and~\ref{l:Ytildemoments}, respectively, we obtain that
\begin{align*}
    \bbE[\norm{w_{xx}(t,\widetilde{Y}(t))}{\cL(H)}]
        &\lesssim \left(1 + \bbE[\norm{\widetilde{Y}(t)}{H}^p]
            + \bbE[\norm{Y(1)}{H}^p]
            + \bbE[\norm{Y(t)}{H}^p]  \right) \\
        &\lesssim \Bigl(1 + e^{-\frac{p\pi^2}{2k}} h^{-\frac{pd}{2}} + \norm{g}{H}^p \Bigr),
\end{align*}
since $\tr(L^{-2\beta}) < \infty$. 
This completes the proof of the weak error estimate~\eqref{e:err-weak}.
\end{proof}

\begin{remark}
Note that, if the first and second Fr\'echet derivatives of $\varphi$
are bounded, the estimates of the Lemmata~\ref{l:Ymoments}
and~\ref{l:Ytildemoments}
are not needed and the weak error estimate in~\eqref{e:err-weak}
simplifies to
\begin{align*}
    \bigl| \bbE[\varphi(u)] &- \bbE[\varphi(u_{h,k}^{Q})] \bigr| \\
         &\leq C
            \Bigl( h^{\min\{d(2\alpha\beta-1),r,s\}} + e^{-\frac{\pi^2}{k}} h^{-d}
                + e^{-\frac{\pi^2}{2k}} + e^{-\frac{\pi^2}{2k}} f_{\alpha,\beta}(h) \Bigr)
                (1 + \norm{g}{\theta} ).
\end{align*}
The calibration of the discretization parameters $k$ and $h$
remains as described in Remark~\ref{remark:calibration}.
\end{remark}

\section{An application and numerical experiments}\label{section:numexp}

In this section we validate the theoretical results
of the previous sections within the scope of a simulation study
based on the model for Mat\'ern approximations in \eqref{e:matern}
on the domain $\cD = (0,1)^d$ for $d=1,2$, $\kappa = 0.5$,
and $u=0$ on $\partial\cD$, i.e., $L=\kappa^2 - \Delta$
with homogeneous Dirichlet boundary conditions.
In this case, the operator $L$
has the following eigenvalue-eigenvector pairs
\cite[Ch.~VI.4]{courant1962}:
\begin{align}\label{e:eig:laplace}
    \lambda_{\jvec} = \kappa^2 + \pi^2 |\jvec|^2 = \kappa^2 + \pi^2 \sum_{i=1}^{d} j_i^2,
    \qquad
    e_{\jvec}(\xvec)
        = \prod_{i=1}^{d} \left( \sqrt{2} \sin(\pi j_i \, x_i) \right),
\end{align}
where $\jvec=(j_1,\ldots,j_d)\in\bbN^d$
is a $d$-dimensional multi-index.
As already mentioned in Example~\ref{ex:matern},
these eigenvalues satisfy~\eqref{e:ass:lambdaj} for $\alpha = 2/d$.

Note that for every $\xvec\in\cD$ the solution
$u$ satisfies $u(\xvec) \sim \cN(0,\sigma(\xvec)^2)$.
Following a Karhunen--Lo\`{e}ve expansion of $u$ with
respect to the eigenfunctions $\{e_{\jvec}\}_{\jvec\in\bbN^d}$
in~\eqref{e:eig:laplace},
the variance $\sigma(\xvec)^2$ can be expressed explicitly
in terms of the eigenvalues and eigenfunctions in~\eqref{e:eig:laplace} by
\begin{align}\label{e:sigma2}
\sigma(\xvec)^2
    &= \bbE\Bigl|\sum_{\jvec \in \bbN^d}
        \lambda_{\jvec}^{-\beta} \widetilde{\xi}_{\jvec} \, e_{\jvec}(\xvec) \Bigr|^2
    = \sum_{\jvec \in \bbN^d}
        \lambda_{\jvec}^{-2\beta} e_{\jvec}(\xvec)^2 ,
\end{align}
where $\bigl\{\widetilde{\xi}_{\jvec}\bigr\}_{\jvec\in\bbN^d}$ are independent
$\cN(0,1)$-distributed random variables.

Considering continuous evaluation functions $\varphi\from L_2(\cD) \to \bbR$ of the form
\begin{align*}
    \varphi(u) = \int_{\cD} f(u(\xvec)) \rd \xvec
\end{align*}
allows us to perform the simulation study
without Monte Carlo sampling, since
\begin{align*}
    \bbE[ \varphi(u) ] &= \int_{\cD} \bbE[ f( u(\xvec) ) ] \rd \xvec,
\end{align*}
and the value of $\bbE[f(u(\xvec))]$ can be derived analytically
from $u(\xvec) \sim \cN(0,\sigma(\xvec)^2)$.
More precisely, we choose $f(u) = |u|^p$,
$p=2,3,4$, as well as $f(u) = \Phi(20(u-0.5))$,
where $\Phi(\cdot)$ denotes
the cumulative distribution function for the standard normal distribution.
The motivation of the latter function is given
by its correspondence to a probit transform
which is often used to approximate step functions (see, e.g.,~\cite{barman2017three}),
in this case $\mathds{1}(u>0.5)$.
These four functions satisfy Assumption~\ref{ass:all}\ref{ass:phi}
and we obtain for the quantity of interest,
\begin{align}\label{e:Ephi1}
\bbE[ \varphi(u) ]
     = \tfrac{2^{p/2}\Gamma((p+1)/2)}{\sqrt{\pi}} \int_{\cD} \sigma(\xvec)^p \rd \xvec,
\end{align}
if  $f(u) = |u|^p$, and
\begin{align}\label{e:Ephi2}
\bbE[ \varphi(u) ]
    = \int_{\cD} \Phi\left( - \tfrac{a}{\sqrt{c^{-2} + \sigma(\xvec)^2}}\right) \rd \xvec,
\end{align}
if $f(u) = \Phi(c(u-a))$ for $a\in\bbR$ and $c>0$.

We truncate the series in \eqref{e:sigma2}
in order to approximate the variance $\sigma(\xvec)^2$,
\begin{align*}
\sigma(\xvec)^2
    &\approx \sum_{j_1=1}^{\Nok}\cdots\sum_{j_d=1}^{\Nok}\lambda_{(j_1,\ldots,j_d)}^{-2\beta}e_{(j_1,\ldots,j_d)}(\xvec)^2.
\end{align*}
Here, we choose $\Nok = 1 + 2^{18}$ for $d=1$
and $\Nok = 1 + 2^{11}$ for $d=2$
so that, in both cases, $\Nok^d \gg N_h$
for all considered finite element spaces
with $N_h$ basis functions.
This estimate of $\sigma(\xvec)$ is used
at $\Nok^d$ equally spaced locations $\xvec\in\cD$,
and the reference solution $\bbE[\varphi(u)]$ is then approximated
by applying the trapezoidal rule in order to
evaluate the integrals in \eqref{e:Ephi1}
and \eqref{e:Ephi2} numerically.

\begin{table}[t]
\centering
\caption{\label{tab:nodenumber}Numbers of finite element basis functions
and the corresponding numbers of quadrature nodes as a function of $\beta$}
\begin{tabular}{lccccc}
\toprule
& & \multicolumn{4}{c}{$\beta$} \\
& $N_h$ & $0.6$  & $0.7$ & $0.8$ & $0.9$ \\
\cmidrule(r){2-6}
\multirow{4}{*}{$d=1$} 	& 511   	& 146  	&  226  	&  386  & 866  \\
				  	& 1023   	& 180 	&  278  	& 476  & 1069 \\  					
				  	& 2047   	& 218  	&  337  	& 576  & 1293 \\ 					
				  	& 4095  	& 258 	& 400	& 685  & 1538 \\ 	
\cmidrule(r){1-6}  				  				
\multirow{4}{*}{$d=2$} 	& 225  	& 24   	& 36    	& 60   &  133  \\
					& 961  	& 38   	& 58    	& 98   & 218 \\
					& 3969 	& 56   	& 86    	& 145 & 325 \\
					& 16129 	& 78 		& 119  	& 203 & 453 \\					
\bottomrule
\end{tabular}
\end{table}

We consider \eqref{e:matern} for $\beta = 0.6,0.7,0.8,0.9$
and use a finite element discretization based on
continuous piecewise linear basis functions
with respect to uniform meshes on $\bar{\cD} = [0,1]^d$.
We use four different mesh sizes $h$ in each dimension $d=1,2$,
and calibrate the quadrature step size $k$
with $h$ for each value of $\beta$
by $k = -1/(\beta\ln h)$.
This results in the numbers of basis functions
and quadrature nodes shown in Table~\ref{tab:nodenumber}.
As already pointed out in Example \ref{ex:matern},
the growth exponent of the eigenvalues
is in this case $\alpha=2/d$,
and Assumption \ref{ass:all} is satisfied for $r=s=q=2$.
This gives the theoretical value $\min\{4\beta-d,2\}$
for the weak convergence rate.

For the computation of $\bbE[\varphi(u_{h,k}^Q)]$
we can use the same procedure as
for the reference solution in order to avoid Monte Carlo simulations.
For this purpose, we have to replace $\sigma(\xvec)^2$
in~\eqref{e:Ephi1} and~\eqref{e:Ephi2}
by the variance of the finite element solution,
$\sigma_{h}(\xvec)^2 = \operatorname{Var}(u_{h,k}^Q(\xvec))$.
To this end, we first assemble the matrix
\begin{align*}
    \Qmat_{h,k}^{\beta}
        = \frac{2 k \sin(\pi\beta)}{\pi} \sum_{\ell =-K^{-}}^{K^{+}} e^{2\beta y_{\ell}} (\Mmat + e^{2 y_{\ell}}(\kappa^2\Mmat + \Smat))^{-1},
\end{align*}
where $y_\ell := \ell k$
and $\Mmat, \Smat \in\bbR^{N_h\times N_h}$
are the mass matrix and the stiffness matrix
with respect to the finite element basis $\{\phi_{j,h}\}_{j=1}^{N_h}$
with entries
\begin{align*}
    M_{ij} := \scalar{\phi_{i,h}, \phi_{j,h}}{L_2(\cD)},
    \quad
    S_{ij} := \scalar{\nabla \phi_{i,h}, \nabla \phi_{j,h}}{L_2(\cD)},
    \quad
    1\leq i,j, \leq N_h.
\end{align*}
If we let
$\phivec_h(\xvec) := (\phi_{1,h}(\xvec),\ldots, \phi_{N_h,h}(\xvec))^{\T}$
denote the vector of the finite element basis functions evaluated at $\xvec\in\cD$
and $\bvec := (\scalar{\white_h^\Phi, \phi_{j,h}}{L_2(\cD)})_{j=1}^{N_h} \sim \cN(\mathbf{0},\Mmat)$,
the variance $\sigma_h(\xvec)^2$ is given by
\begin{align*}
    \sigma_h(\xvec)^2
        = \operatorname{Var}(u_{h,k}^Q(\xvec))
        = \operatorname{Var}\Bigl( \phivec_h(\xvec)^{\T} \Qmat_{h,k}^{\beta} \bvec \Bigr)
        = \phivec_h(\xvec)^\T\Qmat_{h,k}^{\beta}\Mmat(\Qmat_{h,k}^{\beta})^\T\phivec_h(\xvec).
\end{align*}
The computation of $\sigma_h(\xvec)^2$
at the same $\Nok^d$ locations as for the reference solution
again enables a numerical evaluation of
the integrals in~\eqref{e:Ephi1} and~\eqref{e:Ephi2}
via the trapezoidal rule
for approximating $\bbE[\varphi(u_{h,k}^Q)]$.

\begin{figure}[p]
\begin{center}
\begin{minipage}[t]{0.405\linewidth}
\begin{center}
$f(u) = |u|^2$, $d=1$
\includegraphics[width=\linewidth]{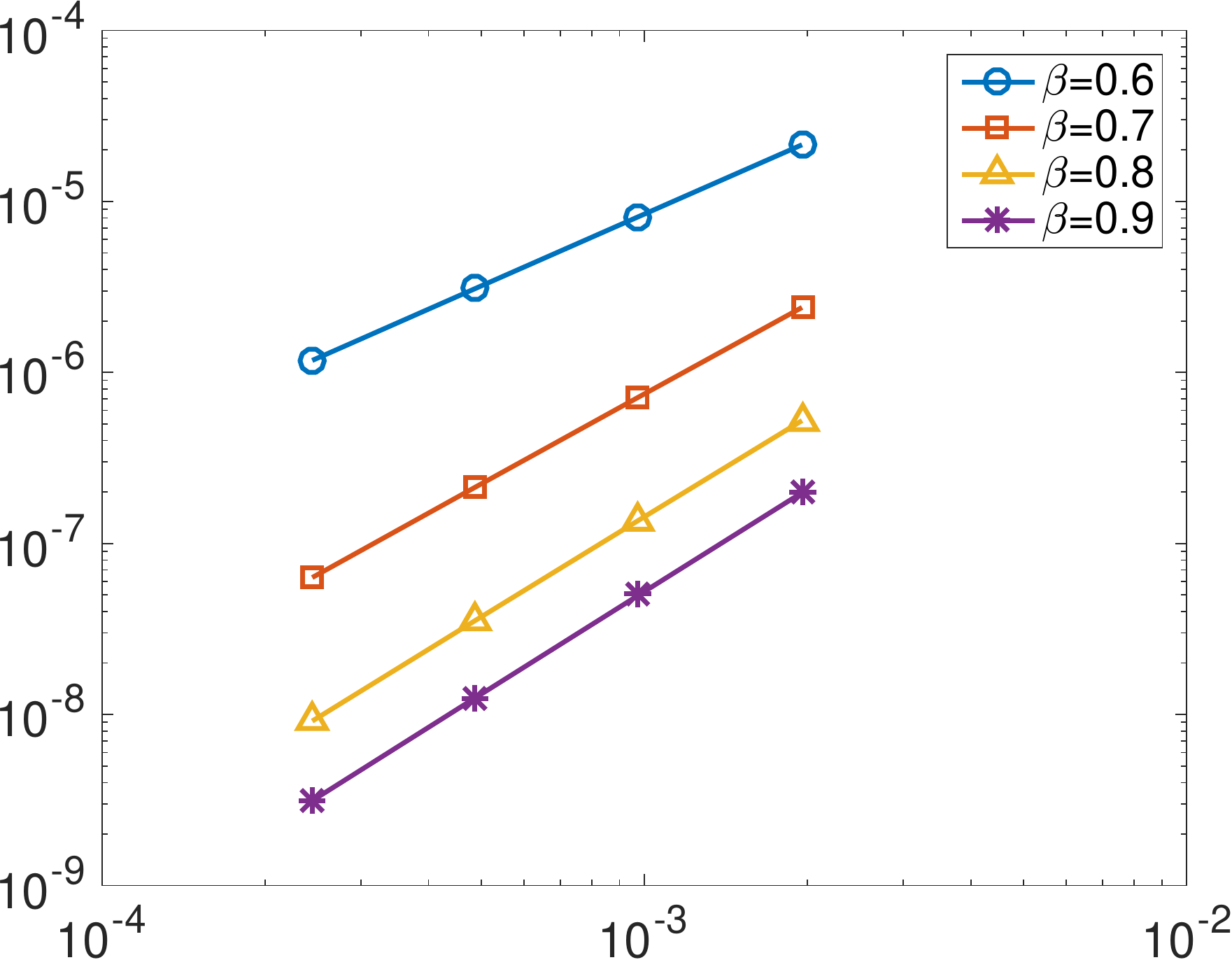}
\end{center}
\end{minipage}
\hspace*{1cm}
\begin{minipage}[t]{0.405\linewidth}
\begin{center}
$f(u) = |u|^2$, $d=2$
\includegraphics[width=\linewidth]{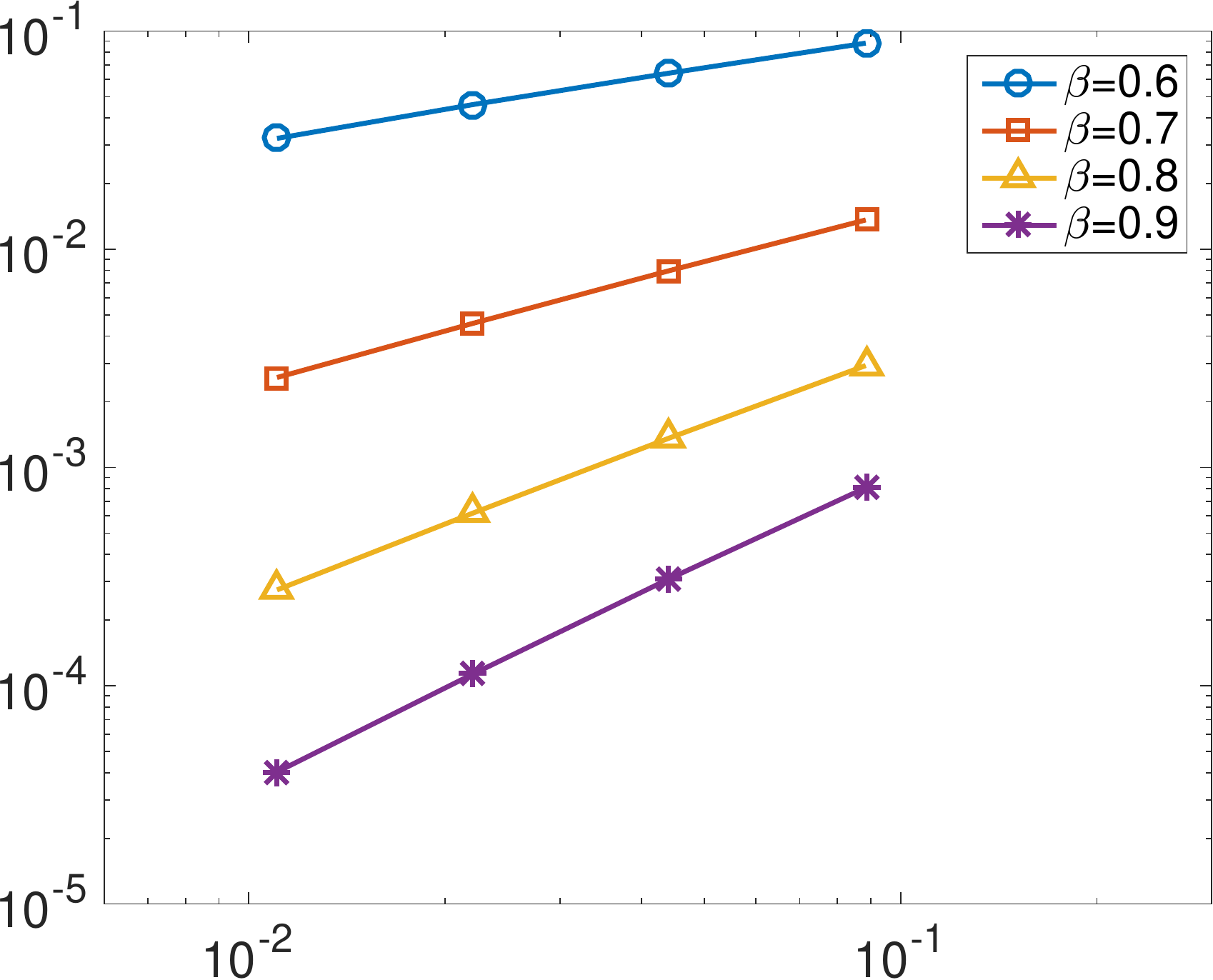}
\end{center}
\end{minipage}
\begin{minipage}[t]{0.405\linewidth}
\begin{center}
$f(u) = |u|^3$, $d=1$
\includegraphics[width=\linewidth]{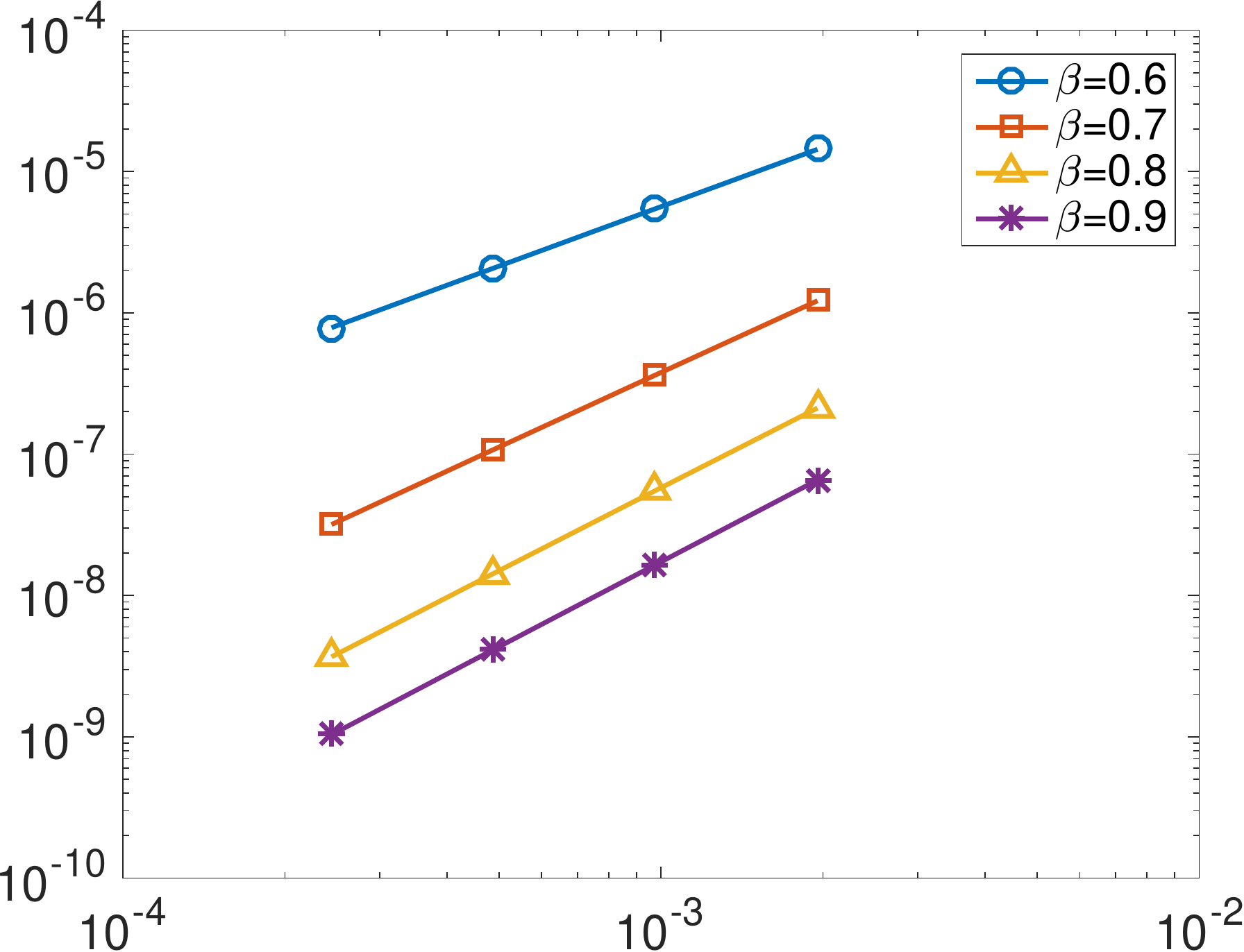}
\end{center}
\end{minipage}
\hspace*{1cm}
\begin{minipage}[t]{0.405\linewidth}
\begin{center}
$f(u) = |u|^3$, $d=2$
\includegraphics[width=\linewidth]{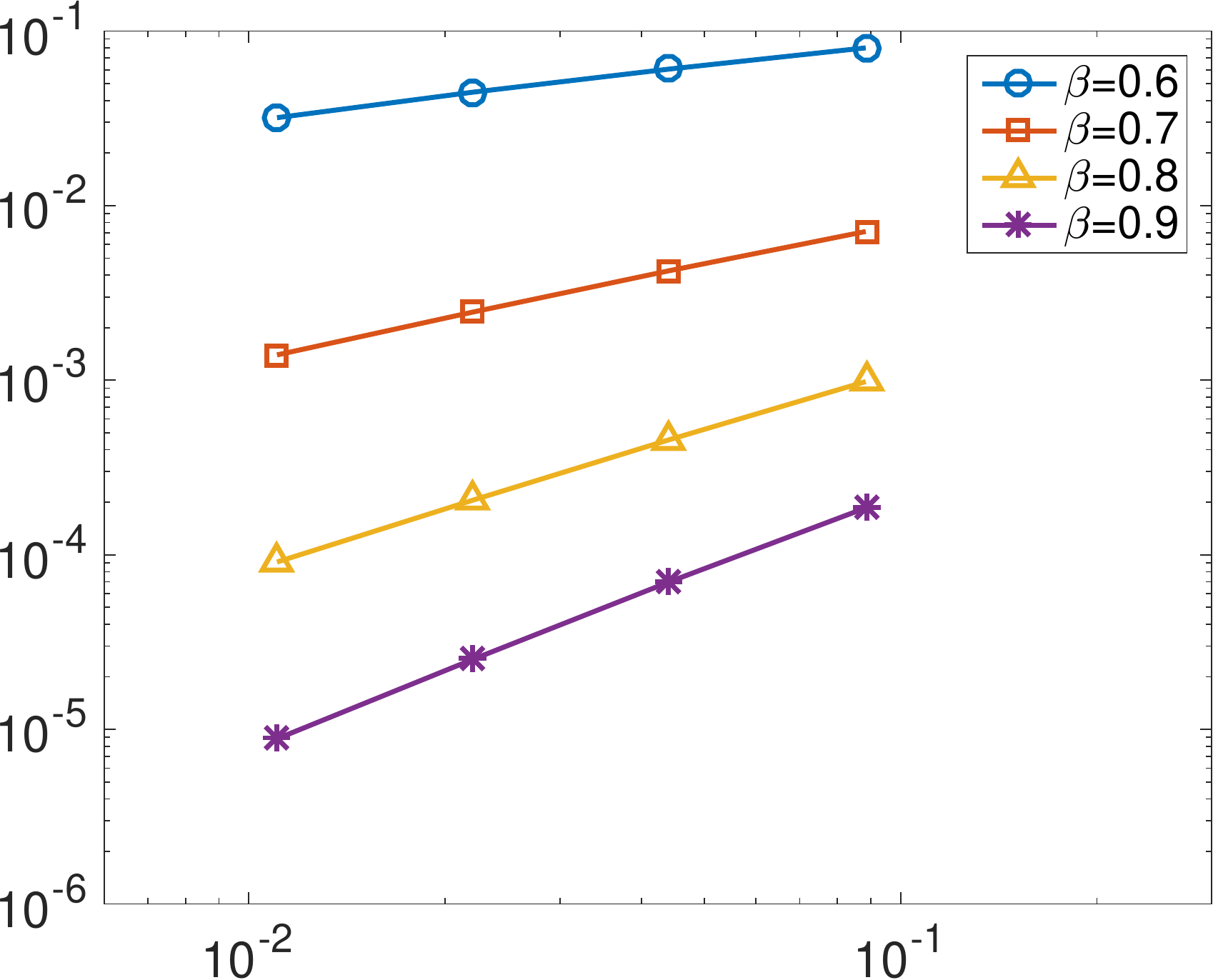}
\end{center}
\end{minipage}
\begin{minipage}[t]{0.405\linewidth}
\begin{center}
$f(u) = |u|^4$, $d=1$
\includegraphics[width=\linewidth]{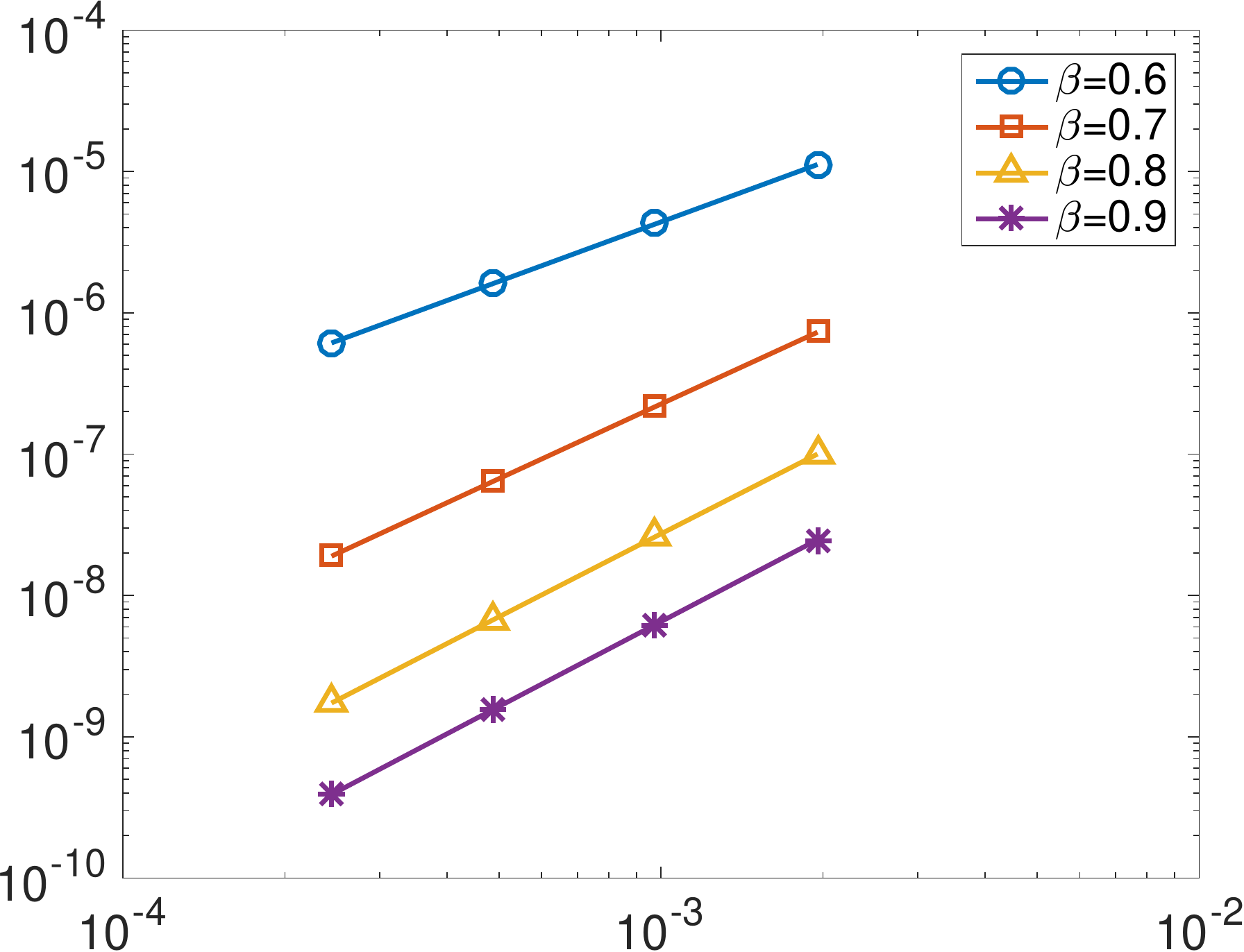}
\end{center}
\end{minipage}
\hspace*{1cm}
\begin{minipage}[t]{0.405\linewidth}
\begin{center}
$f(u) = |u|^4$, $d=2$
\includegraphics[width=\linewidth]{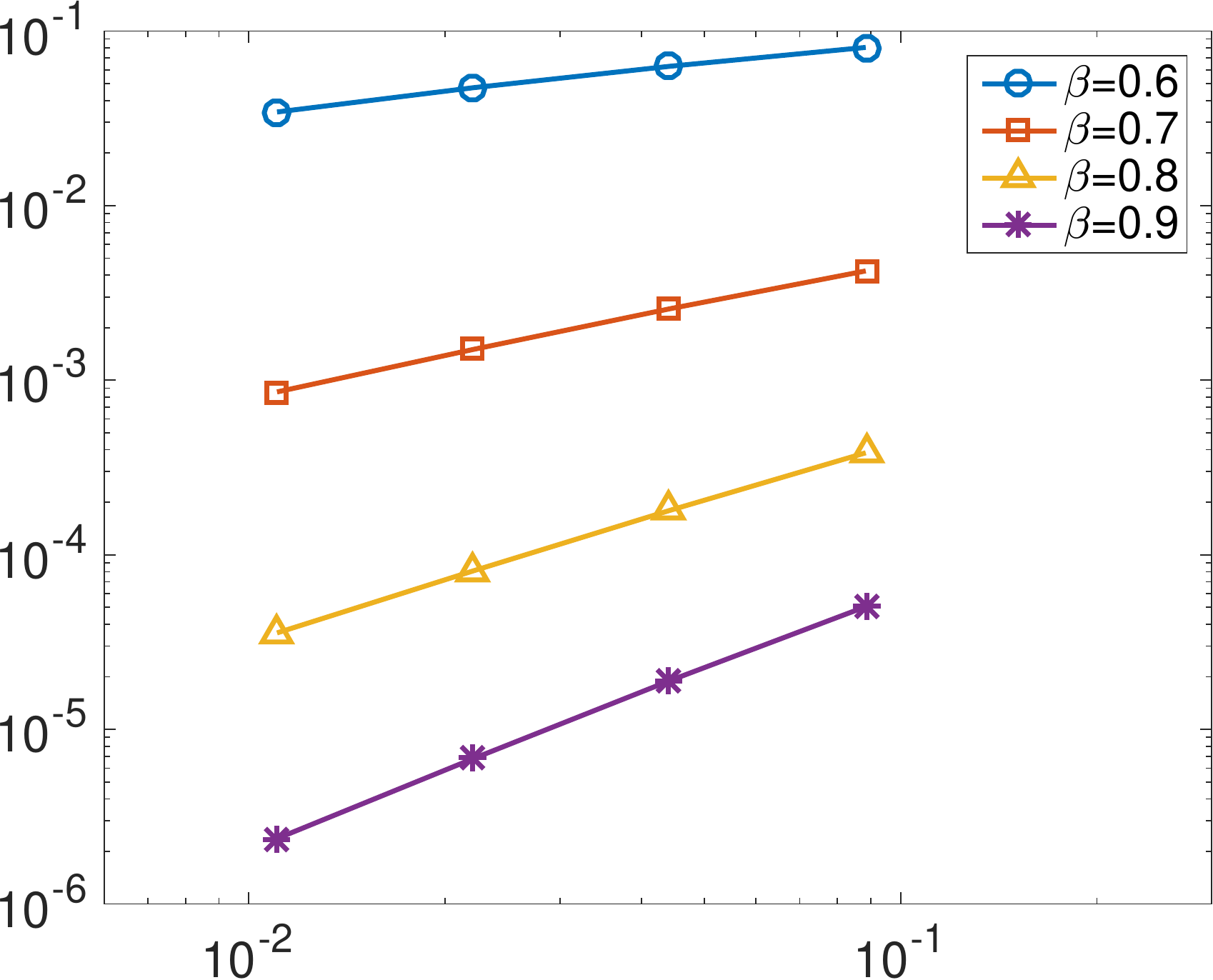}
\end{center}
\end{minipage}
\begin{minipage}[t]{0.405\linewidth}
\begin{center}
$f(u) = \Phi(20(u-0.5))$, $d=1$
\includegraphics[width=\linewidth]{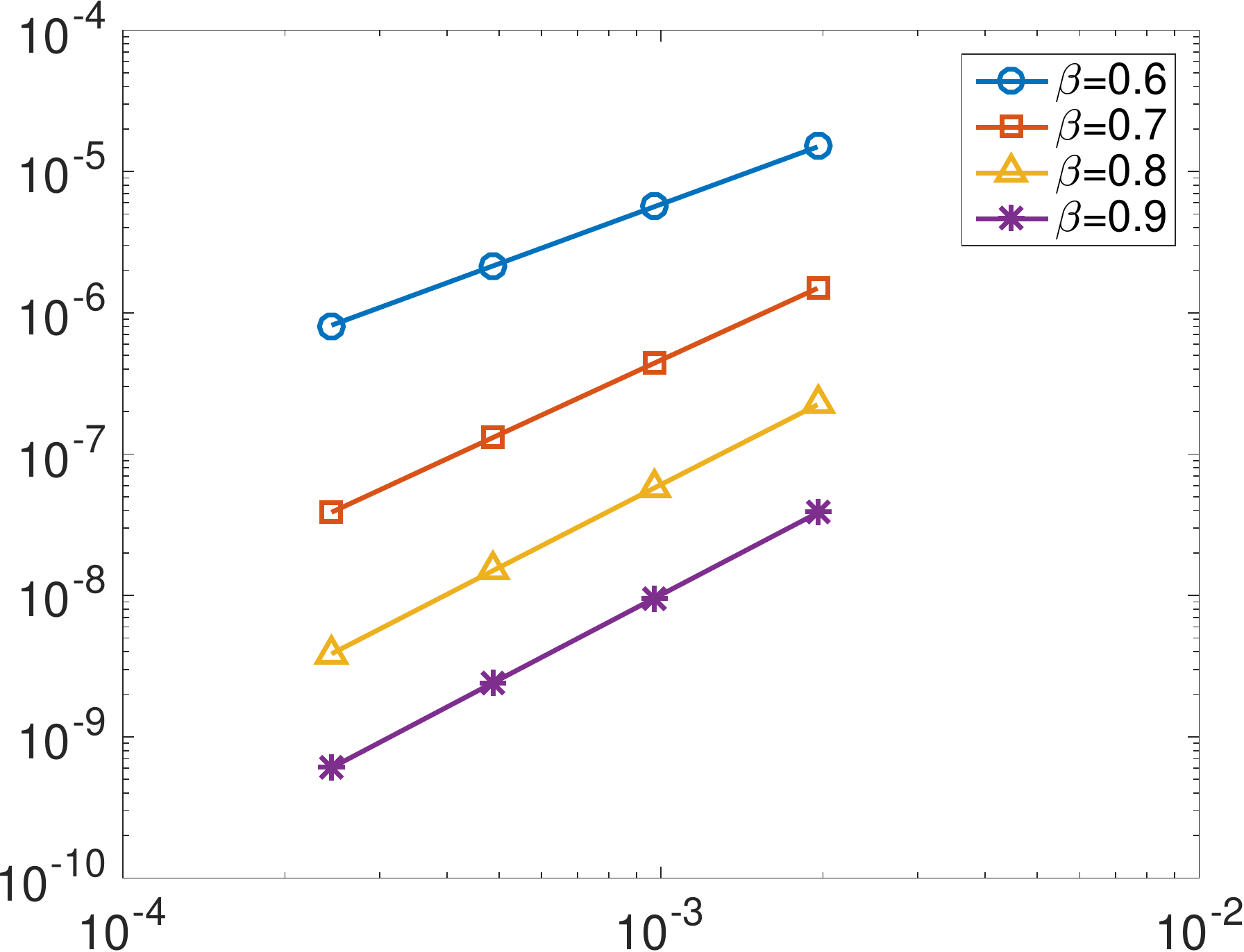}
\end{center}
\end{minipage}
\hspace*{1cm}
\begin{minipage}[t]{0.405\linewidth}
\begin{center}
$f(u) = \Phi(20(u-0.5))$, $d=2$
\includegraphics[width=\linewidth]{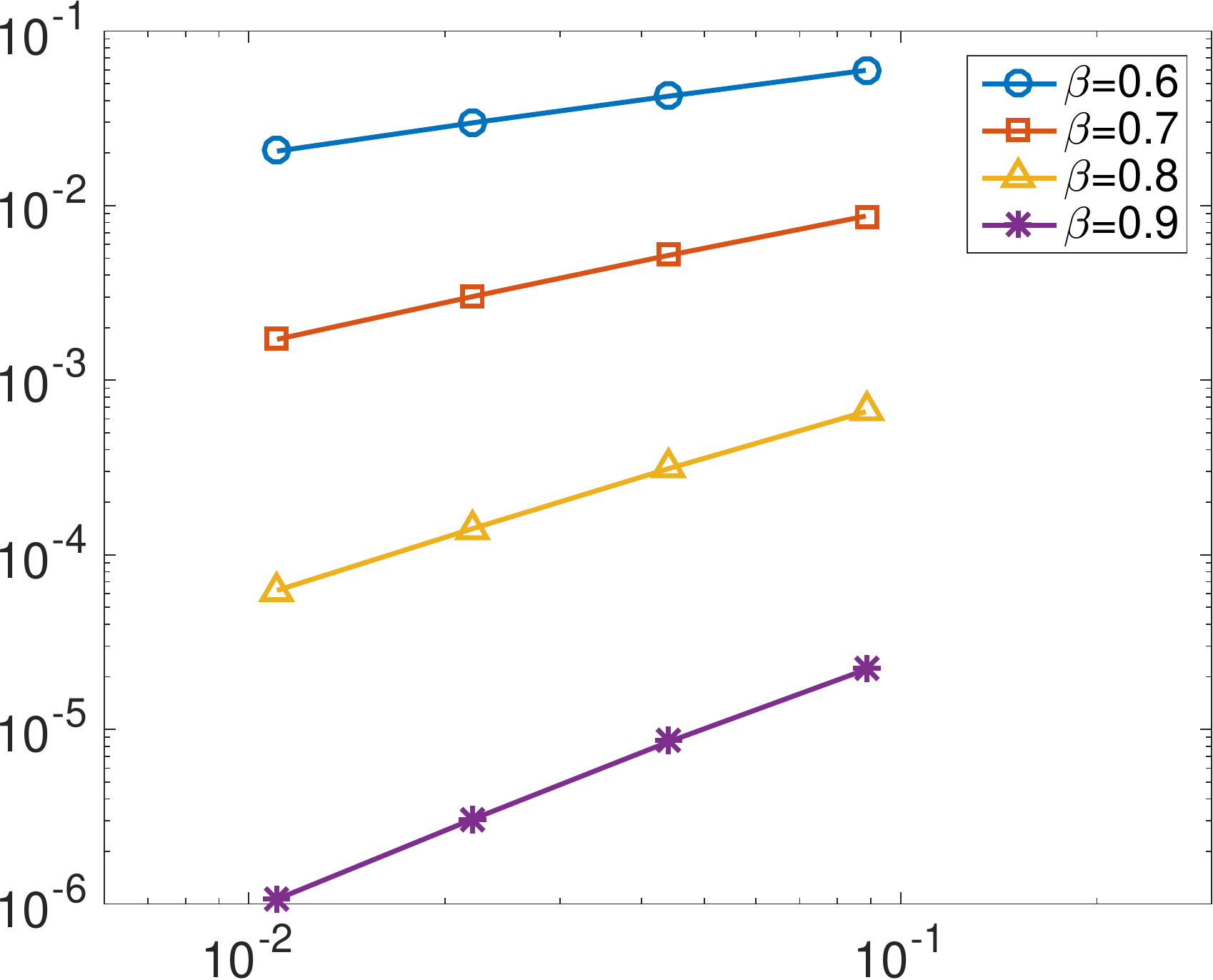}
\end{center}
\end{minipage}
\end{center}
\caption{Observed weak errors for $d=1,2$ and different values of $\beta$.
The errors for the four choices of $\varphi(u) = \int_{\cD} f(u(\xvec)) \rd \xvec$ are shown as
functions of the mesh size $h$ in a log-log scale.
The corresponding observed convergence rates are shown in Table~\ref{tab:rates}.}
\label{fig:errors}
\end{figure}

The resulting observed weak errors
$\mathrm{err}_{\ell} := |\bbE[\varphi(u)] - \bbE[\varphi(u_{h_{\ell},k}^Q)]|$
are shown in Figure \ref{fig:errors}.
For each function $\varphi$ and for each value of $\beta$,
we compute the empirical convergence rate $\mathrm{r}$ by a least-squares fit
of a line $\mathrm{c} + \mathrm{r} \ln h$ to the data set $\{h_{\ell}, \mathrm{err}_{\ell}\}$.
The results are shown in Table \ref{tab:rates}
and can be seen to validate the theoretical rates given in Theorem \ref{thm:weak-conv} for $d=1$.
For $d=2$, the observed rates deviate slightly
from the theoretical rates for $\beta=0.9$,
which is caused by the fact that we had to use
coarser finite element meshes for $d=2$ than for $d=1$
in order to be able to assemble the dense matrices
$\Qmat_{h,k}^{\beta} \in \bbR^{N_h\times N_h}$
for performing the simulation study
without Monte Carlo simulations.

\begin{table}[t]
\centering
\caption{\label{tab:rates}Observed (resp.~theoretical) rates of convergence for the weak errors shown in Figure \ref{fig:errors}}
\begin{tabular}{lccccc}
\toprule
& &  \multicolumn{4}{c}{$\beta$}\\
					& $f(u)$			&  $0.6$  		& $0.7$ 		& $0.8$ 		& $0.9$\\
\cmidrule(r){2-6}
\multirow{3}{*}{$d=1$}	& $|u|^2$  		& 1.396 (1.4) 	& 1.748 (1.8) 	& 1.945 (2) 	& 1.994 (2) \\
				  	& $|u|^3$  		& 1.397 (1.4) 	& 1.753 (1.8) 	& 1.949 (2) 	& 1.995 (2) \\
					& $|u|^4$  		& 1.398 (1.4) 	& 1.754 (1.8) 	& 1.951 (2) 	& 1.996 (2) \\
					& $\Phi(20(u-0.5))$  	& 1.398 (1.4) 	& 1.755 (1.8) 	& 1.952 (2) 	& 1.996 (2) \\
\cmidrule(r){1-6}
\multirow{3}{*}{$d=2$}	& $|u|^2$  		& 0.483 (0.4)	& 0.800 (0.8) 	& 1.139 (1.2)  	& 1.442 (1.6)\\
					& $|u|^3$  		& 0.442 (0.4)	& 0.783 (0.8) 	& 1.145 (1.2)  	& 1.465 (1.6)\\
					& $|u|^4$  		& 0.409 (0.4)	& 0.768 (0.8) 	& 1.143 (1.2)  	& 1.472 (1.6)\\
					& $\Phi(20(u-0.5))$  	& 0.512 (0.4) 	& 0.782 (0.8) 	& 1.135 (1.2) 	& 1.458 (1.6) \\
\bottomrule
\end{tabular}
\end{table}

\section{Conclusion}\label{section:conclusion}

Gaussian random fields are of great importance as models in spatial statistics.
A popular method for reducing the computational cost
for operations, which are needed during statistical inference,
is to represent the Gaussian field as a solution to an SPDE.
In this work, we have investigated a recent extension of this approach
to Gaussian random fields with general smoothness proposed in \cite{bkk17}.
The method considers the fractional order equation \eqref{e:Lbeta}
and is based on combining
a finite element discretization in space
with the quadrature approximation~\eqref{e:def:Qhk} of the inverse fractional power operator.
This yields an approximate solution $u_{h,k}^Q$ of the SPDE,
which in \cite{bkk17} was shown to converge
to the solution $u$ of~\eqref{e:Lbeta} in the strong mean-square sense
with rate~\eqref{e:rate-strong}.

In many applications one is mostly interested
in a certain quantity of the random field $u$
which can be expressed by $\varphi(u)$
for some real-valued function $\varphi$.
For this reason, the focus of the present work
has been the weak error
$|\bbE[\varphi(u)] - \bbE[\varphi(u_{h,k}^Q)] |$.
The main outcome of this article, Theorem \ref{thm:weak-conv},
shows convergence of this type of error to zero
at an explicit rate
for twice continuously Fr\'echet differentiable
functions~$\varphi$, which have a second derivative
of polynomial growth.
Notably, the component of the
convergence rate stemming from the
stochasticity of the problem
is doubled compared to the
strong convergence rate~\eqref{e:rate-strong} derived in \cite{bkk17}.
For proving this result, we have performed
a rigorous error analysis in \S\ref{sec:derivation},
which is based on an extension of the equation~\eqref{e:Lbeta}
to a time-dependent problem as well as
an associated Kolmogorov backward equation and It\^{o} calculus.

In order to validate
the theoretical findings,
we have performed a simulation study
for the stochastic model problem \eqref{e:matern}
on the domain $\cD = (0,1)^d$ for $d=1,2$
in \S\ref{section:numexp}.
This model is highly relevant for applications in spatial statistics,
since it is often used to approximate Gaussian Mat\'ern fields.
We have considered four different functions $\varphi$
and the fractional orders $\beta = 0.6, 0.7, 0.8, 0.9$.
The observed empirical weak convergence rates
can be seen to verify the theoretical results.
One of the considered functions $\varphi$
is based on a transformation of the random field
by a Gaussian cumulative distribution function.
Quantities of this form are particularly important
for applications to porous materials, as they are used
to model the pore volume fraction of the material, see, e.g., \cite{barman2017three}.
Thus, we see ample possibilities for applying the outcomes
of this work to problems in spatial statistics and related disciplines.

\bibliographystyle{siam}
\bibliography{bkk-bib}

\end{document}